\documentclass[11pt]{article}
\usepackage{a4}
\usepackage{amsmath,amssymb,amsthm}
\usepackage{mathrsfs}
\usepackage{amsfonts}
\usepackage{url,xcolor,graphicx}
\usepackage{makecell}
\usepackage{dsfont}
\usepackage{multirow}
\usepackage{comment}
\usepackage{mathrsfs}

\newtheorem{theorem}{Theorem}[section]
\newtheorem{lemma}[theorem]{Lemma}

\newtheorem{assumption}[theorem]{Assumption}
\newtheorem{corollary}[theorem]{Corollary}
\newtheorem{example}[theorem]{Example}
\newtheorem{remark}[theorem]{Remark}

\numberwithin{equation}{section}

\newcommand{\eps}{\varepsilon}
\newcommand {\beq} {\begin{equation}}
\newcommand {\eeq} {\end{equation}}

\begin{document}

\title{Uniform error analysis of a rectangular Morley finite element method on a Shishkin mesh for a 4th-order singularly perturbed boundary value problem}

\author{
\normalsize Xiangyun Meng$^\dagger$,~~Martin Stynes$^\ast$\\ \normalsize $^\dagger$ School of Mathematics and Statistics,
Beijing Jiaotong University, \\
\normalsize Beijing 100044, China\\
\normalsize $^\ast$ Applied and Computational Mathematics Division, Beijing Computational Science Research Center,\\
\normalsize Beijing 100193, China\\\vspace{2mm}
\normalsize email: xymeng1@bjtu.edu.cn; ~~  m.stynes@csrc.ac.cn(corresponding author);\normalsize
}
\date{}

\maketitle

\begin{abstract}
The singularly perturbed reaction-diffusion problem
$\eps^2\Delta^2 u - \mathrm{div}\left(c\nabla u\right) = f$ is considered on
the unit square~$\Omega$ in~$\mathbb{R}^2$
with homogenous Dirichlet boundary conditions. Its solution typically contains boundary layers on
all sides of~$\Omega$. It is discretised by a finite element method that uses
rectangular Morley elements on a Shishkin mesh.
In an associated energy-type norm that is natural for this problem, we prove an
$O(\eps^{1/2}N^{-1}+\eps N^{-1}\ln N + N^{-3/2})$ rate of convergence
for the error in the computed solution,
where $N$~is the number of mesh intervals in each coordinate direction.
Thus in the most troublesome regime when $\eps \approx N^{-1}$,
our method is proved to attain an $O(N^{-3/2})$ rate of convergence,
which is shown to be sharp by our numerical experiments
and is superior to the $O(N^{-1/2})$ rate that is proved in
Meng \& Stynes, Adv. Comput. Math. 2019
when Adini finite elements are used to solve the same problem on the same mesh.
\end{abstract}

\section{Introduction}\label{sec:intro}
Set $\Omega:=(0,1)^2 \subset \mathbb{R}^2$, with closure~$\bar\Omega$ and boundary~$\partial\Omega$.
The bending of a simply supported plate is modelled \cite{Frankbook,GLN12,Semper92}
by the following 4th-order singularly perturbed boundary value problem
in which $u$ represents the deflection of the plate under the transverse loading $f$:
\begin{subequations}\label{BVP}
\begin{align}
\eps^2\Delta^2 u(x,y) - \mathrm{div}\left(c(x,y)\nabla u(x,y)\right)
	&= f(x,y) \quad \mbox{in}\ \Omega, \label{BVPeq}
\\
u(x,y) = \frac{\partial u(x,y)}{\partial n} &= 0 \quad\mbox{on}\ \partial\Omega, \label{BVPbdry}
\end{align}
\end{subequations}
where the parameter $\eps$ (which is the ratio of bending rigidity to tensile stiffness of the plate) satisfies $0<\eps\ll 1$, the function $c$ lies in $C(\bar\Omega)$ with  $c\ge c_0>0$ for some constant~$c_0$ and $f$~lies in $C^1(\bar\Omega)$.
The vector $n = (n_x,n_y)^T$ is the outward-pointing unit normal vector to the boundary $\partial\Omega$ of~$\Omega$,
so $\partial u(x,y)/\partial n$ denotes the normal derivative on~$\partial\Omega$.

During the past two decades, various finite element methods (FEMs) have ben proposed
for the numerical solution of 4th-order singularly perturbed problems.
Error analyses on quasi-uniform meshes using nonconforming FEMs are discussed in
\cite{ChenShi05,GLN12,Huang21,LiMing23,Winther01,Xie13,WangXuHu06},
and a $C^0$ interior penalty method in~\cite{Brenner11}.
Because it is troublesome to impose global $C^1$ continuity on the computed solution in the 2D domain~$\Omega$ for 4th-order problems,
the $C^0$ interior penalty method relaxes the global $C^1$ continuity by adding penalty terms \cite[eq.(2.3)]{Brenner11}  to a globally $C^0$ finite element space.
Similarly, the discontinuous Galerkin method  adds stabilisation terms \cite[(3.4) and (3.13)]{FKMC07}
to the finite element space without imposing global $C^0$~continuity.
Compared with these two approaches, nonconforming FEMs discard the global $C^1$~continuity condition but do not usually add any stabilisation terms.

The well-known nonconforming triangular Morley element for 4th-order problems
was constructed in \cite{Morley68} and appears in many textbooks  \cite{BrennerScottbook,Ciarletbook,GanesanTobiskabook,ShiWang}.
This element does not however yield convergence for 2nd-order boundary value problems
and is therefore unsuitable for our singularly perturbed problem~\eqref{BVP};
this issue is discussed in~\cite[Section~1]{Winther01}.
Thus to solve~\eqref{BVP}, various modifications of the triangular Morley  element
have been proposed \cite{Winther01,WangXuHu06}.

An alternative departure from the triangular Morley element is
the \emph{rectangular} Morley element which was introduced
in~\cite{WangShiXu07} and \cite[Section 2.7.1]{ShiWang}. It uses the same degrees of freedom
as the triangular Morley element, viz., the point value at each node
and the integral mean of the outer normal derivative on each edge.
But unlike its triangular counterpart, the rectangular Morley element is convergent for both 4th-order and 2nd-order problems, so it is reasonable to consider it for solving~\eqref{BVP},
as has already been pointed out in~\cite{Xie13}.

Note also that the rectangular Morley element has only eight degrees of freedom in each element,
which compares well with the twelve degrees of freedom of the nonconforming Adini element
and the sixteen d.o.f. of the conforming Bogner-Fox-Schmit element (these two finite elements are also suitable for solving the problem~\eqref{BVP}).
We shall use the rectangular Morley element method  to solve~\eqref{BVP}.

To solve \eqref{BVP} effectively, one must also use an appropriate mesh.
Since the exact solution of~\eqref{BVP} exhibits boundary layers (see Assumption~\ref{ass:S} below),
the mesh should be refined near~$\partial\Omega$.
Suitable meshes were used with an Adini nonconforming FEM in~\cite{MengStynes19},
a $C^0$ interior penalty method in~\cite{Franz14},
a mixed FEM in~\cite{Franz16} and a mixed $hp$-FEM in~\cite{Constantinou19}.
One sees easily that the meshes most suitable for~\eqref{BVP} are severely anisotropic,
so one cannot use results from the literature that assume quasi-uniformity of the mesh,
but mesh quasi-uniformity is assumed in many of the earlier papers that we mentioned earlier.
Consequently, for example, while Lemmas~\ref{consis1-1} and~\ref{consis2-1} below
were proved previously assuming mesh quasi-uniformity,  we have to prove them
on meshes without this assumption.

We shall use a tensor-product Shishkin mesh when solving~\eqref{BVP}.
In an energy-type norm that is natural for this problem,
we prove $O(\eps^{1/2}N^{-1}+\eps N^{-1}\ln N + N^{-3/2})$ convergence,
where $N$~is the number of mesh elements in each coordinate direction
and this convergence is uniform in the small parameter~$\eps$.
This result is much better than the $O(N^{-1/2})$ rate that is attained
on quasi-uniform meshes \cite[Corollary 4.1]{Xie13}.

Furthermore, comparing our Morley result with the nonconforming
Adini element studied in~\cite{MengStynes19},
in the most troublesome regime when $\eps \approx N^{-1}$,
our method achieves optimal $O(N^{-3/2})$ convergence,
which is better than the  $O(N^{-1/2})$ attained in~\cite{MengStynes19}; see Remark \ref{AvM}.

The paper is structured as follows.
Section~\ref{sec:decomp} gives a decomposition of the solution~$u$
of~\eqref{BVP} that is needed for the entire analysis.
The rectangular Morley element is defined in Section~\ref{sec:Morley}.
In Section~\ref{sec:discreteproblem} we define the discrete analogue of~\eqref{BVP} that is solved
by our FEM, and we derive various preliminary results for this discretisation.
The analysis in this section is valid on arbitrary tensor-product meshes.
The Shishkin mesh is defined in Section~\ref{sec:Shishkin}, then we continue our earlier analysis of the
numerical method, finally proving its convergence in Theorem~\ref{thm:cgce}.
The paper is completed by numerical experiments in Section~\ref{sec:numer}.
\\[2mm]

\noindent\emph{Notation.} Throughout the paper, $C$ denotes a generic constant that is independent of $\eps$
and of the mesh.  Standard Lebesgue spaces $L^p(\Omega')$ and Sobolev spaces $H^k(\Omega')$ are used,
where $\Omega'$ is any measurable subset of $\Omega$.
Their associated norms are $\|\cdot\|_{L^2(\Omega')}$, $\|\cdot\|_{H^k(\Omega')}$
and the associated semi-norm is $\left|\cdot\right|_{H^k(\Omega')}$.
The inner product in $L^2(\Omega')$ is denoted by $(\cdot,\cdot)_{\Omega'}$, but we omit the subscript when $\Omega' = \Omega$.

\subsection{A decomposition of the solution $u$}\label{sec:decomp}
If $c\in C(\bar\Omega)\subset L^2(\Omega)$ and $f\in C^1(\bar\Omega)\subset L^2(\Omega)$,
then (cf.~\cite[Theorem~5.1]{GanesanTobiskabook}) one can use the Lax-Milgram lemma to show that   \eqref{BVP} has a unique solution in~$H^2(\Omega)$,
but for our analysis we need more precise information about the behaviour of~$u$,
especially near the boundary~$\partial\Omega$.

Number the sides of $\bar\Omega$ in clockwise order as $1,2,3,4$, where $1$ is associated with the interval $[0,1]$ on the $x$-axis. Then denote the corner where the sides $i$ and~$j$ meet by  $(i;j)$ with $i<j$.
As in the papers \cite{Franz16,MengStynes19}, we make the following assumption.

\begin{assumption}\label{ass:S}
The solution $u$ of the boundary value problem~\eqref{BVP} can be decomposed as
\begin{align*}
u(x,y) &= S(x,y) + \sum_{i = 1}^4 E_i(x,y) + E_{12}(x,y) + E_{23}(x,y) + E_{34}(x,y) + E_{14}(x,y)\\
&:= S(x,y) + u^E(x,y)
\end{align*}
for $(x,y)\in\bar\Omega$,
where $S(x,y)$ is a smooth function,
each $E_i(x,y)$  is a boundary layer component associated with the side $i$ of~$\bar\Omega$,
and each $E_{ij}(x,y)$ is a corner layer component associated with the corner~$(i;j)$, i.e.,
there exists a constant $C$ such that for all $(x,y)\in\overline{\Omega}$ and $0\le i+j\le 4$ one has
\begin{subequations}\label{apriori}
\begin{align}
\left|\frac{\partial^{i+j}S(x,y)}{\partial x^i\partial y^j}\right|&\le C, \label{priS}\\
\left|\frac{\partial^{i+j}E_1(x,y)}{\partial x^i\partial y^j}\right|&\le C\eps^{1-j}e^{-y/\eps}, \label{priE1}\\
\left|\frac{\partial^{i+j}E_{12}(x,y)}{\partial x^i\partial y^j}\right|&\le C\eps^{1-i-j}e^{-x/\eps}e^{-y/\eps}, \label{priE12}
\end{align}
\end{subequations}
with analogous bounds for the remaining components of the decomposition.
\end{assumption}

\section{The rectangular Morley element on anisotropic meshes}\label{sec:Morley}
For the moment we use an arbitrary tensor product mesh on~$\Omega$.
In Section~\ref{sec:Shishkin} we shall specialise our analysis to a suitable Shishkin mesh.

Let $N$ be a positive integer.
Let $0 = x_0 < x_1  \dots < x_N = 1$ and $0 = y_0 < y_1  \dots < y_N = 1$ be arbitrary meshes
in the interval $[0,1]$ on the $x$ and $y$ axes respectively.
Draw vertical and horizontal lines through these points to partition $\bar\Omega$ into mesh rectangular elements with
vertices $\{(x_i,y_j): 0\le i,j\le N\}$. We write $\mathcal{T}_N$ for the union of these mesh rectangles, so  $\mathcal{T}_N$ is a partition of~$\bar\Omega$.

Let $K$ be a typical mesh rectangle as shown in Figure~\ref{fig:DOF}.
Its vertices are $a_1, a_2, a_3, a_4$ and its edges are $e_1, e_2, e_3,e_4$ with unit normal vectors $n_{e_j}$ for $j=1,2,3,4$.

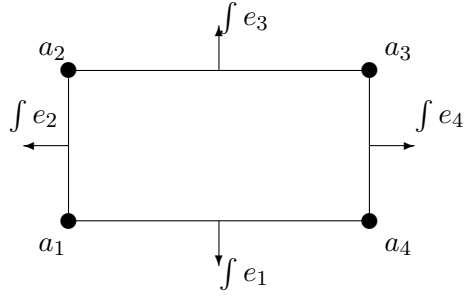
\begin{figure}[htbp]\unitlength0.5pt
\begin{center}
\setlength{\unitlength}{2cm}
\begin{picture}(2,2)

\put(0,0.5){\line(1,0){2}}
\put(0,0.5){\line(0,1){1}}
\put(2,0.5){\line(0,1){1}}
\put(0,1.5){\line(1,0){2}}

\put(0,0.5){\circle*{0.1}}
\put(0,1.5){\circle*{0.1}}
\put(2,0.5){\circle*{0.1}}
\put(2,1.5){\circle*{0.1}}

\put(1,0.5){\vector(0,-1){0.3}}
\put(1,1.5){\vector(0,1){0.3}}
\put(0,1){\vector(-1,0){0.3}}
\put(2,1){\vector(1,0){0.3}}

\put(1,0.1){$\int e_1$}
\put(1,1.8){$\int e_3$}
\put(-0.4,1.15){$\int e_2$}
\put(2.3,1.15){$\int e_4$}

\put(-0.2,0.3){$a_1$}
\put(2.1,0.3){$a_4$}
\put(2.1,1.6){$a_3$}
\put(-0.2,1.6){$a_2$}

\linethickness{0.6mm}
\end{picture}
\end{center}
\caption{degrees of freedom on element $K$}
\label{fig:DOF}
\end{figure}

Let $(x_{c},y_{c})$ be the center of $K$ and suppose $K$ has dimensions $2h_x \times 2h_y$.
Then the coordinates of the vertices $a_j$ are $a_1=(x_c-h_x,y_c-h_y)$, $a_2=(x_c-h_x,y_c+h_y)$,
$a_3=(x_c+h_x,y_c+h_y)$, $a_4=(x_c+h_x,y_c-h_y)$.
For $j=1,2,3,4$, let  $|e_j|$ denote the length of edge $e_j$ and write $n_{e_j}$ for the outward-pointing unit normal to $e_j$.

We now describe the rectangular Morley element of~ \cite{ShiXieNMPDE10,ShiWang,WangShiXu07}.
For each mesh rectangle~$K$, let $P(K)$ denote the span of $\{1,x,y,x^2,xy,y^2,x^3,y^3\}$.

Define a set of 8 functions in $P(K)$ as follows:
\begin{equation*}
\begin{aligned}
&p_1(x,y) = \frac{1}{8}\left[2\left(1-\frac{x-x_{c}}{h_x}\right)\left(1-\frac{y-y_{c}}{h_y}\right)
+\left(\frac{x-x_{c}}{h_x}\right)^3-\frac{x-x_{c}}{h_x}
+\left(\frac{y-y_{c}}{h_y}\right)^3-\frac{y-y_{c}}{h_y}\right],\\
&p_2(x,y)  = \frac{1}{8}\left[2\left(1-\frac{x-x_{c}}{h_x}\right)\left(1+\frac{y-y_{c}}{h_y}\right)
+\left(\frac{x-x_{c}}{h_x}\right)^3-\frac{x-x_{c}}{h_x}
-\left(\frac{y-y_{c}}{h_y}\right)^3+\frac{y-y_{c}}{h_y}\right],\\
&p_3(x,y)  = \frac{1}{8}\left[2\left(1+\frac{x-x_{c}}{h_x}\right)\left(1+\frac{y-y_{c}}{h_y}\right)
-\left(\frac{x-x_{c}}{h_x}\right)^3+\frac{x-x_{c}}{h_x}
-\left(\frac{y-y_{c}}{h_y}\right)^3+\frac{y-y_{c}}{h_y}\right],\\
&p_4(x,y)  = \frac{1}{8}\left[2\left(1+\frac{x-x_{c}}{h_x}\right)\left(1-\frac{y-y_{c}}{h_y}\right)
-\left(\frac{x-x_{c}}{h_x}\right)^3+\frac{x-x_{c}}{h_x}
+\left(\frac{y-y_{c}}{h_y}\right)^3-\frac{y-y_{c}}{h_y}\right],\\
&q_1(x,y)  = -\frac{h_y}{4}\left(\frac{y-y_{c}}{h_y}+1\right)\left(\frac{y-y_{c}}{h_y}-1\right)^2,
\ \  q_2(x,y) = -\frac{h_x}{4}\left(\frac{x-x_{c}}{h_x}+1\right)\left(\frac{x-x_{c}}{h_x}-1\right)^2,\\
&q_3(x,y)  = \frac{h_y}{4}\left(\frac{y-y_{c}}{h_y}+1\right)^2\left(\frac{y-y_{c}}{h_y}-1\right),
\ \  q_4(x,y) = \frac{h_x}{4}\left(\frac{x-x_{c}}{h_x}+1\right)^2\left(\frac{x-x_{c}}{h_x}-1\right).
\end{aligned}
\end{equation*}
One can verify (see~\cite[eq.(2.5)]{MengYangZhangSCM16}) that
\begin{equation}\label{nodeedge}
\begin{cases}
p_i(a_j)=\delta_{ij} \ \text{ and }\ \frac{1}{|e_j|}\int_{e_j}\frac{\partial p_i}{\partial {n}_{e_j}}\mathrm{d}s=0
	&\text{ for }i,j=1,2,3,4,\\
q_i(a_j)=0\ \text{ and }\ \frac{1}{|e_j|}\int_{e_j}\frac{\partial q_i}{\partial {n}_{e_j}}\mathrm{d}s=\delta_{ij}
	&\text{for }i,j=1,2,3,4,
\end{cases}
\end{equation}
where $\delta_{ij}$ is the Kronecker symbol.
Hence the $p_i$ and $q_j$ are linearly independent and form a basis for~$P(K)$.

The \emph{rectangular Morley element space $V_N$} is now defined to be the set of functions $v \in L^2(\Omega)$
such that  $v|_K \in P(K)$ for each mesh rectangle~$K$,
$v$~is continuous at all mesh vertices in~$\Omega$ as one moves from one mesh rectangle to another,
$v$~vanishes at all vertices in~$\partial\Omega$,
and for all $K$ one has $\int_e \left[\frac{\partial v}{\partial n_e}\right]\,\mathrm{d}s =0$ for all edges~$e$ of~$K$
where $n_e$~is the unit normal to~$e$ that points out of~$K$;
here $\left[\frac{\partial v}{\partial n_e}\right]$ denotes the jump in $\frac{\partial v}{\partial n_e}$ as one moves
from one side of~$e$ to the other side (i.e., from one mesh rectangle to another) if  $e\not\subset\partial\Omega$,
while $\left[\frac{\partial v}{\partial n_e}\right] = \frac{\partial v}{\partial n}$ if $e\subset\partial\Omega$.

\section{The discrete problem and its error analysis}\label{sec:discreteproblem}
In this section we discretise \eqref{BVP} using a finite element method based on the rectangular Morley element
and begin the error analysis of this method.

\subsection{Discretisation of \eqref{BVP}}\label{sec:discretisation}
Any function $v_N\in V_N$ may be only piecewise continuous on~$\Omega$ so one cannot apply
differential operators (such as $\nabla$ and $\Delta$) to $v_N$ globally. Thus, we apply these operators
only piecewise on the interiors of the elements $K\in \mathcal{T}_N$ and indicate this modification by writing
$\nabla_N$ and $\Delta_N$ instead of $\nabla$ and $\Delta$, e.g., $(\nabla_N v_N)|_K  = \nabla (v_N|_K)$.
For $v_N, w_N\in V_N$, define the bilinear forms
$b_N(v_N,w_N):=(c\nabla_N v_N,\nabla_N w_N)$ and $a_N(v_N,w_N):=(\Delta_N v_N,\Delta_N w_N)$.
We shall also use expressions like $b_N(u,w_N)$ and $a_N(u,w_N)$, which are defined in the same way.

One can write \eqref{BVP} in the equivalent weak form
\[
\eps^2(\Delta u, \Delta v) + \left(c\nabla u, \nabla v\right) =  (f,v) \ \text{ for all } v \in H_0^2(\Omega),
\]
where as usual $H_0^2(\Omega) = \big\{g\in H^2(\Omega): g|_{\partial\Omega} = \frac{\partial g}{\partial n}|_{\partial\Omega} = 0\big\}$.

Then our discretisation of~\eqref{BVP} is: Find $u_N \in V_N$ such that
\begin{equation}
\eps^2 a_N(u_N,v_N)+b_N(u_N,v_N)=(f,v_N)_{L^2(\Omega)}\;\quad  \forall v_N \in V_N. \label{BVP_FEM}
\end{equation}
Existence and uniqueness of the solution $u_N$ of \eqref{BVP_FEM} follows readily from the Lax-Milgram lemma.

\subsection{Preliminaries}\label{sec:prelim}

For each $K\in{\mathcal T}_N$,
let $P_0(K)$ be the space of constant functions defined on~$K$ and
let $P_K^0: L^2(K)\rightarrow P_0(K)$ be the $L^2$ projector, viz.,
\begin{equation*}
(v,w_N)_K = (P_K^0 v,w_N)_{K}\quad \forall v\in L^2(K), \, \forall w_N\in P_0(K).
\end{equation*}
From~\cite[Lemma~7.5]{Georgoulis06}, for all $v\in H^1(K)$ one has the anisotropic error bound
\begin{align}\label{L2pro}
\left\|v-P_K^0v \right\|_{L^2(K)}^2\le C\left(h_x^2\left\|\frac{\partial v}{\partial x} \right\|_{L^2(K)}^2+h_y^2\left\|\frac{\partial v}{\partial y} \right\|_{L^2(K)}^2\right),
\end{align}
with $C$ independent of $v$ and $K$.

Let $Q_1(K)$ denote the usual space of polynomials of the form $c_0+c_1x+c_2y+c_3xy$ defined on~$K$.
Set $Q_N=\big\{v_N\in H^1(\Omega): v_N|_K\in Q_1(K) \,\forall K\in \mathcal{T}_N\big\}$.
Let $I_N: C(K)\rightarrow Q_N$ denote the standard piecewise bilinear nodal interpolation operator,
viz., given $v\in C(K)$, for each vertex $P$ of $K$ we set $I_N v(P) = v(P)$.
Then for each $K$ and all $v\in H^1(K) \cap C(K)$, by \cite[Theorem 2.7]{Apelbook} one has
\begin{align}\label{L2inter}
\left\|v-I_Nv \right\|_{L^2(K)}^2\le C\left(h_x^2\left\|\frac{\partial v}{\partial x} \right\|_{L^2(K)}^2+h_y^2\left\|\frac{\partial v}{\partial y} \right\|_{L^2(K)}^2\right)
\end{align}
and
\begin{align}\label{H1stab}
\left\|I_N v \right\|_{H^1(K)}\le C\left\|v \right\|_{H^1(K)}.
\end{align}

By applying a standard scaling argument to the $x$ and $y$ variables separately,
one can modify the proof of \cite[Lemma~4.5.3]{BrennerScottbook}
to obtain the following anisotropic inverse inequalities:
for each $K\in{\mathcal T}_N$ and any $\xi_N \in P(K)$, one has
\begin{subequations}
\begin{align}
h_x\left\|\frac{\partial^2 \xi_N}{\partial x^2}\right\|_{L^2(K)}
	+ h_y\left\|\frac{\partial^2 \xi_N}{\partial x\partial y}\right\|_{L^2(K)}
		&\le C\left\|\frac{\partial \xi_N}{\partial x}\right\|_{L^2(K)},\label{inverse1}\\
h_y\left\|\frac{\partial^2 \xi_N}{\partial y^2}\right\|_{L^2(K)}
	+ h_x\left\|\frac{\partial^2 \xi_N}{\partial x\partial y}\right\|_{L^2(K)}
		&\le C\left\|\frac{\partial \xi_N}{\partial y}\right\|_{L^2(K)},\label{inverse2}
\end{align}
\end{subequations}
where the constants $C$ are independent of $K$ and $\xi_N$.

\subsection{The error equation}
Let $v_N\in V_N$ be arbitrary. Set $\xi_N = u_N-v_N$ and $\eta_N = u-v_N$.
Then \eqref{BVP_FEM} and \eqref{BVP} yield
\begin{align*}
\eps^2 a_N(\xi_N,\xi_N) &+ b_N(\xi_N,\xi_N)  \notag \\
& = \eps^2 a_N(\eta_N,\xi_N)+\eps^2 a_N(u_N-u,\xi_N) + b_N(\eta_N,\xi_N) + b_N(u_N-u,\xi_N)  \notag \\
&= \eps^2 a_N(\eta_N,\xi_N) + b_N(\eta_N,\xi_N)+(\eps^2\Delta^2 u
	- \mathrm{div}(c\nabla u), I_N \xi_N)\notag \\
&\qquad+(f,\xi_N-I_N \xi_N)-\eps^2 a_N(u,\xi_N)-b_N(u,\xi_N).
\end{align*}
Integrating by parts and using $I_N\xi_N=0$ on $\partial\Omega$,
as in \cite[(3.4)]{HuShiYang16} we get
\[
(\Delta^2 u,I_N\xi_N)= -(\nabla\Delta u,\nabla (I_N \xi_N) )  \text{ and }
-(\mathrm{div}(c\nabla u),I_N\xi_N)-b_N(u, \xi_N)=b_N(u,I_N \xi_N-\xi_N).
\]
Integration by parts on each $K\in \mathcal{T}_N$ gives, like \cite[(3.5)]{HuShiYang16},
\begin{align*}
&a_N(u,\xi_N)=\sum_{K\in \mathcal{T}_N}(\Delta u,\Delta \xi_N)_{L^2(K)}\notag\\
&\quad= -\sum_{K\in \mathcal{T}_N}(\nabla\Delta u,\nabla \xi_N)_{L^2(K)}+\sum_{K\in \mathcal{T}_N}\int_{\partial K}\frac{\partial^2 u}{\partial n^2}\frac{\partial \xi_N}{\partial n}\;\mathrm{d}s
+\sum_{K\in \mathcal{T}_N}\int_{\partial K}\frac{\partial ^2u}{\partial s\partial n}\frac{\partial \xi_N}{\partial s}\;\mathrm{d}s,
\end{align*}
where $\partial/\partial s$ and $\partial/\partial n$ are, respectively, the tangential and
outward-pointing normal derivatives
along the boundary $\partial K$.
From these calculcations we get the \emph{error equation}
\begin{align}
\eps^2 a_N(\xi_N,\xi_N) &+ b_N(\xi_N,\xi_N)  \notag\\
 &= \eps^2 a_N(\eta_N,\xi_N) + b_N(\eta_N,\xi_N)-b_N(u,\xi_N-I_N \xi_N)+(f,\xi_N-I_N \xi_N)\notag\\
&\qquad-\eps^2\left[\sum_{K\in \mathcal{T}_N}\int_{\partial K}\frac{\partial^2 u}{\partial n^2}\frac{\partial \xi_N}{\partial n}\;\mathrm{d}s+\sum_{K\in \mathcal{T}_N}\int_{\partial K}\frac{\partial ^2u}{\partial s\partial n}\frac{\partial \xi_N}{\partial s}\;\mathrm{d}s\right]\notag\\
&\qquad+\eps^2\sum_{K\in \mathcal{T}_N}\int_{K} (\nabla\Delta u) \nabla (\xi_N-I_N \xi_N)\;\mathrm{d}x\,\mathrm{d}y. \label{errorequation}
\end{align}

\begin{remark}
The error equation~\eqref{errorequation} is closely related to~\cite[(3.7)]{Xie13};
see also~\cite[(3.6)]{HuShiYang16}.
The analysis of~\eqref{errorequation} that we now develop
has some similarity to the analysis in~\cite{Xie13},
but our work is greatly complicated (compared with~\cite{Xie13})
by the extreme anisotropy of the Shishkin mesh
and by our aim of proving error bounds
that show convergence even when $\eps$~is very small;
in~\cite{Xie13} the mesh is shape-regular and the bounds established in
Theorems~4.1 and~4.2 blow up if one takes the limit as~$\eps\to 0$,
while the bound of \cite[Theorem~4.3]{Xie13} (which remains valid as $\eps\to 0$)
is only $O(N^{-1/2})$, which is inferior to the $O(N^{-3/2})$ bound
that we shall prove in Theorem~\ref{thm:cgce}.
\end{remark}

\subsection{Bounds on error equation integrals along $\partial K$}
In this subsection we bound the terms
$
\eps^2\sum_{K\in \mathcal{T}_N}\int_{\partial K}\frac{\partial^2 u}{\partial n^2}\frac{\partial \xi_N}{\partial n}\;\mathrm{d}s$
and
$\eps^2\sum_{K\in \mathcal{T}_N}\int_{\partial K}\frac{\partial ^2u}{\partial s\partial n}\frac{\partial \xi_N}{\partial s}\;\mathrm{d}s$
from the error equation \eqref{errorequation}.

For all $v_N\in V_N$, define $|v_N|_{1,N}^2:=(\nabla_N v_N,\nabla_N v_N), \, |v_N|_{2,N}^2:=(\Delta_N v_N,\Delta_N v_N)$
and the norm $\|v_N\|_{\eps,N}^2:=\eps^2|v_N|_{2,N}^2+|v_N|_{1,N}^2$.
One can define $|u|_{1,N}, |u|_{2,N}$ and $\|u\|_{\eps,N}$ in the same way.

Let the edges of each $K\in \mathcal{T}_N$ be labelled $e_i$ for $i=1,\dots,4$ as in Figure~\ref{fig:DOF}.
For each $w\in L^2(e_i)$ define
\[
\Pi^0_{e_i}w=\frac{1}{|e_i|}\int_{e_i}w\,\mathrm{d}s
\ \text{ and }\
\mathcal{R}^0_{e_i}w=w-\Pi^0_{e_i}w.
\]

The argument used in the next lemma imitates in part the proof of \cite[Lemma~3.2]{HuShiYang16},
where only a uniform mesh was considered.

\begin{lemma}\label{consis1-1}
Assume that $w\in H_0^2(\Omega)\cap H^4(\Omega)$.
Then there exists a constant $C$ such that for all $\xi_N\in V_N$, one has
\begin{align}
&\left|\sum_{K\in \mathcal{T}_N}\int_{\partial K}
	\frac{\partial^2 w}{\partial n^2}\frac{\partial \xi_N}{\partial n}\;\mathrm{d}s\right| \notag\\
&\qquad\le C \left[\left(\sum_{K\in \mathcal{T}_N}h_x^4\left\|\frac{\partial^4 w}{\partial x\partial y^3}\right\|_{L^2(K)}^2\right)^{1/2}
	+\left(\sum_{K\in \mathcal{T}_N}h_y^4\left\|\frac{\partial^4 w}{\partial x^3\partial y}\right\|_{L^2(K)}^2\right)^{1/2}\right]
		|\xi_N|_{2,N}  \label{consis1-1-eq}
\intertext{and}
&\left|\sum_{K\in \mathcal{T}_N}\int_{\partial K}
	\frac{\partial^2 w}{\partial n^2}\frac{\partial \xi_N}{\partial n}\;\mathrm{d}s\right| \notag\\
&\qquad\le C \left[\sum_{K\in \mathcal{T}_N}\left(h_x^2\left\|\frac{\partial^4 w}{\partial x\partial y^3}\right\|_{L^2(K)}^2\right)^{1/2}+\sum_{K\in \mathcal{T}_N}\left(h_y^2\left\|\frac{\partial^4 w}{\partial x^3\partial y}\right\|_{L^2(K)}^2\right)^{1/2}\right] |\xi_N|_{1,N}. \label{consis1-2-eq}
\end{align}
\end{lemma}
\begin{proof}
For any internal edge~$e_i$,  the edge-jump condition in the definition of~$V_N$
implies that $\Pi^0_{e_i}\frac{\partial \xi_N}{\partial n}$ has the same magnitude but different signs
when calculated on the two adjacent mesh rectangles~$K$ that share~$e_i$.
Consequently $\sum_{K\in \mathcal{T}_N}\sum_{i=1}^4\int_{e_i}\frac{\partial^2 w}{\partial n^2}
	\left(\Pi^0_{e_i}\frac{\partial \xi_N}{\partial n}\right)\;\mathrm{d}s =0$,
as in these integrals $\frac{\partial^2 w}{\partial n^2}$ depends on~$e_i$ on each internal edge  but not on the associated~$K$,
while if $e_i\subset\partial\Omega$ then $\frac{\partial \xi_N}{\partial n}=0$.
Hence
\begin{align}
\sum_{K\in \mathcal{T}_N}\int_{\partial K}\frac{\partial^2 w}{\partial n^2}\frac{\partial \xi_N}{\partial n}\;\mathrm{d}s
&= \sum_{K\in \mathcal{T}_N}\sum_{i=1}^4\int_{e_i}\frac{\partial^2 w}{\partial n^2}
	\frac{\partial \xi_N}{\partial n}\;\mathrm{d}s \notag\\
&= \sum_{K\in \mathcal{T}_N}\sum_{i=1}^4\int_{e_i}\frac{\partial^2 w}{\partial n^2}
	\left(\mathcal{R}^0_{e_i}\frac{\partial \xi_N}{\partial n}\right)\;\mathrm{d}s
=: \sum_{K\in \mathcal{T}_N}\sum_{i=1}^4J_i,\ \text{say.}  \label{Ji}
\end{align}
	
We now imitate the proof of  \cite[Theorem~4.1]{ShiMaoChen05} by considering $|J_2+J_4|$
and also borrow some techniques from \cite[Lemma~3.2]{HuShiYang16}.
For any $K\in \mathcal{T}_N$, from Figure~\ref{fig:DOF} one sees that
$\frac{\partial \xi_N}{\partial n}=-\frac{\partial \xi_N}{\partial x}$ on $e_2$,
$\frac{\partial \xi_N}{\partial n}=\frac{\partial \xi_N}{\partial x}$ on $e_4$
and $\frac{\partial^2 w}{\partial n^2} = \frac{\partial^2 w}{\partial x^2}$ on $e_2$ and $e_4$.
Thus
\begin{align}
J_2+J_4&=\int_{e_2}\frac{\partial^2 w}{\partial n^2}
	\mathcal{R}^0_{e_2}\frac{\partial \xi_N}{\partial n}\;\mathrm{d}s
		+\int_{e_4}\frac{\partial^2 w}{\partial n^2}
		\mathcal{R}^0_{e_4}\frac{\partial \xi_N}{\partial n}\;\mathrm{d}s  \notag\\
&=-\int_{e_2}\frac{\partial^2 w}{\partial x^2}
	\mathcal{R}^0_{e_2}\frac{\partial \xi_N}{\partial x}\;\mathrm{d}y
	+\int_{e_4}\frac{\partial^2 w}{\partial x^2}
	\mathcal{R}^0_{e_4}\frac{\partial \xi_N}{\partial x}\;\mathrm{d}y  \notag\\
&=\int_{y_{c}-h_y}^{y_{c}+h_y}
	\left[\left(\frac{\partial^2 w}{\partial x^2}
	\mathcal{R}^0_{e_4}\frac{\partial \xi_N}{\partial x}\right)(x_c+h_x,y)
	- \left(\frac{\partial^2 w}{\partial x^2}
	\mathcal{R}^0_{e_2}\frac{\partial \xi_N}{\partial x}\right)(x_c-h_x,y) \right]\,\mathrm{d}y.
	\label{J2J4a}
\end{align}
As $\xi_N\in V_N = \text{span} \{1,x,y,x^2,xy,y^2,x^3,y^3\}$,
it is easy to see that
\begin{equation*}
\left(\mathcal{R}^0_{e_2}\frac{\partial \xi_N}{\partial x}\right)(x_c-h_x,y)
	= \left(\mathcal{R}^0_{e_4}\frac{\partial \xi_N}{\partial x}\right)(x_c+h_x,y)
		\ \text{ for }y_{c}-h_y \le y \le y_{c}+h_y ,
\end{equation*}
and the definition of $\mathcal{R}^0_{e_2}$ implies that
$\int_{y_c-h_y}^{y_c+h_y}\left(\mathcal{R}^0_{e_2}\frac{\partial \xi_N}{\partial x}\right)(x_c-h_x,y)\, \mathrm{d}y=0.$
Hence
\begin{align*}
J_2+J_4
&= \int_{y_{c}-h_y}^{y_{c}+h_y}
	\left[\frac{\partial^2 w}{\partial x^2}(x_c+h_x,y) - \frac{\partial^2 w}{\partial x^2}(x_c-h_x,y) \right]
		\left(\mathcal{R}^0_{e_2}\frac{\partial \xi_N}{\partial x}\right)(x_c-h_x,y) \,\mathrm{d}y \notag\\
=  \int_{y_{c}-h_y}^{y_{c}+h_y}&
	\left[\left( \mathcal{R}^0_{e_4}\frac{\partial^2 w}{\partial x^2}\right)(x_c+h_x,y)
		- \left( \mathcal{R}^0_{e_2}\frac{\partial^2 w}{\partial x^2}\right)(x_c-h_x,y) \right]
		\left(\mathcal{R}^0_{e_2}\frac{\partial \xi_N}{\partial x}\right)(x_c-h_x,y) \,\mathrm{d}y.
\end{align*}

A Cauchy-Schwarz inequality now yields
\begin{align}
(J_2+J_4)^2
&\le \left\{\int_{y_{c}-h_y}^{y_{c}+h_y}
	\left[\left( \mathcal{R}^0_{e_4}\frac{\partial^2 w}{\partial x^2}\right)(x_c+h_x,y)
		- \left( \mathcal{R}^0_{e_2}\frac{\partial^2 w}{\partial x^2}\right)(x_c-h_x,y) \right]^2 \,\mathrm{d}y\right\} \notag\\
&\qquad \times
	\int_{y_{c}-h_y}^{y_{c}+h_y} \left(\mathcal{R}^0_{e_2}\frac{\partial \xi_N}{\partial x}\right)^2(x_c-h_x,y) \,\mathrm{d}y.
	\label{J2J4}
\end{align}
In the second integral here, one has
\begin{align*}
\mathcal{R}^0_{e_2}\frac{\partial \xi_N}{\partial x}(x_c-h_x,y)
&=\frac{1}{2h_y}\int_{y_{c}-h_y}^{y_{c}+h_y}
	\left[\frac{\partial \xi_N}{\partial x}(x_c+h_x,y)-\frac{\partial \xi_N}{\partial x}(x_c+h_x,t)\right]\mathrm{d}t\\
&=\frac{1}{2h_y}\int_{t=y_{c}-h_y}^{y_{c}+h_y}\int_{s=t}^{y}\frac{\partial^2 \xi_N}{\partial x\partial s}(x_{c}+h_x,s)
	\,\mathrm{d}s\,\mathrm{d}t,
\end{align*}
so another Cauchy-Schwarz inequality gives
\begin{align*}
\left[ \mathcal{R}^0_{e_2}\frac{\partial \xi_N}{\partial x}(x_c-h_x,y) \right]^2
	&\le  \frac{1}{4h_y^2} \left\{ \int_{t=y_{c}-h_y}^{y_{c}+h_y}\int_{s=t}^{y}
		\left[\frac{\partial^2 \xi_N}{\partial x\partial s}(x_{c}+h_x,s)\right]^2
			\,\mathrm{d}s\,\mathrm{d}t \right\}
		\left\{ \int_{t=y_{c}-h_y}^{y_{c}+h_y}\int_{s=t}^{y}1\,\mathrm{d}s\,\mathrm{d}t \right\} \notag\\
&\le \frac{h_y}{h_x} \left\|\frac{\partial^2 \xi_N}{\partial x\partial y}\right\|_{L^2(K)}^2,
\end{align*}
where we used the property that $\frac{\partial^2 \xi_N}{\partial x\partial s}$ is independent of $x$,
which follows quickly from the definition of the Morley shape function space $P(K)$ in Section~\ref{sec:Morley}.
Hence
\begin{align}\label{part1}
\int_{y_{c}-h_y}^{y_{c}+h_y} \left(\mathcal{R}^0_{e_2}\frac{\partial \xi_N}{\partial x}\right)^2(x_c-h_x,y) \,\mathrm{d}y
\le  \frac{2h_y^2}{h_x} \left\|\frac{\partial^2 \xi_N}{\partial x\partial y}\right\|_{L^2(K)}^2.
\end{align}
For the first integrand in \eqref{J2J4}, one has
\begin{align*}
\left( \mathcal{R}^0_{e_4}\frac{\partial^2 w}{\partial x^2}\right)(x_c+h_x,y)
	&- \left( \mathcal{R}^0_{e_2}\frac{\partial^2 w}{\partial x^2}\right)(x_c-h_x,y) \\
	&=\frac{\partial^2 w}{\partial x^2}(x_c+h_x,y)-\frac{1}{2h_y}\int_{y_{c}-h_y}^{y_{c}+h_y}
		\frac{\partial^2 w}{\partial x^2}(x_c+h_x,t)\;\mathrm{d}t\\
&\qquad-\frac{\partial^2 w}{\partial x^2}(x_c-h_x,y)+\frac{1}{2h_y}\int_{y_{c}-h_y}^{y_{c}+h_y}\frac{\partial^2 w}{\partial x^2}(x_c-h_x,t)\;\mathrm{d}t\\
&=\frac{1}{2h_y}\int_{y_{c}-h_y}^{y_{c}+h_y}\int_{t}^{y}\frac{\partial^3 w}{\partial x^2\partial s}(x_c+h_x,s)\;\mathrm{d}s\,\mathrm{d}t\\
&\qquad-\frac{1}{2h_y}\int_{y_{c}-h_y}^{y_{c}+h_y}\int_{t}^{y}\frac{\partial^3 w}{\partial x^2\partial s}(x_c-h_x,s)
	\,\mathrm{d}s\,\mathrm{d}t\\
&=\frac{1}{2h_y}\int_{x_{c}-h_x}^{x_{c}+h_x}\int_{y_{c}-h_y}^{y_{c}+h_y}\int_{t}^{y}\frac{\partial^4 w}{\partial x^3\partial s}
	(x,s)\,\mathrm{d}s\,\mathrm{d}t\,\mathrm{d}x.
\end{align*}
Again appealing to Cauchy-Schwarz, we get
\begin{align*}
&\hspace{-7mm}\left[ \left( \mathcal{R}^0_{e_4}\frac{\partial^2 w}{\partial x^2}\right)(x_c+h_x,y)
	- \left( \mathcal{R}^0_{e_2}\frac{\partial^2 w}{\partial x^2}\right)(x_c-h_x,y) \right]^2 \\
&\le \frac{1}{4h_y^2} \left[\int_{x_{c}-h_x}^{x_{c}+h_x}\int_{y_{c}-h_y}^{y_{c}+h_y}\int_{t}^{y}
	\left(\frac{\partial^4 w}{\partial x^3\partial s}(x,s)\right)^2\,\mathrm{d}s\,\mathrm{d}t\,\mathrm{d}x\right]
	\left[\int_{x_{c}-h_x}^{x_{c}+h_x}\int_{y_{c}-h_y}^{y_{c}+h_y}\int_{t}^{y}1\,\mathrm{d}s\,\mathrm{d}t\,\mathrm{d}x\right] \\
&\le 4h_x h_y \left\|\frac{\partial^4 w}{\partial x^3\partial y}\right\|_{L^2(K)}^2;
\end{align*}
hence
\begin{align}
&\hspace{-2cm}\int_{y_{c}-h_y}^{y_{c}+h_y} \left[ \left( \mathcal{R}^0_{e_4}\frac{\partial^2 w}{\partial x^2}\right)(x_c+h_x,y)
	- \left( \mathcal{R}^0_{e_2}\frac{\partial^2 w}{\partial x^2}\right)(x_c-h_x,y) \right]^2 \,\mathrm{d}y \notag\\
&\le 8 h_x h_y^2 \left\|\frac{\partial^4 w}{\partial x^3\partial y}\right\|_{L^2(K)}^2.  \label{part2}
\end{align}
Substituting \eqref{part1} and \eqref{part2} into  \eqref{J2J4} then taking a square root yields
\begin{equation}
\left|J_2+J_4\right| \le 4 h_y^2 \left\|\frac{\partial^4 w}{\partial x^3\partial y}\right\|_{L^2(K)}
	\left\|\frac{\partial^2 \xi_N}{\partial x\partial y}\right\|_{L^2(K)}.  \label{step1}
\end{equation}

One can show similarly that
\begin{equation}\label{step2}
\left|J_1+J_3\right| \le 4 h_x^2 \left\|\frac{\partial^4 w}{\partial x\partial y^3}\right\|_{L^2(K)}
	\left\|\frac{\partial^2 \xi_N}{\partial x\partial y}\right\|_{L^2(K)}.
\end{equation}

Combining \eqref{Ji},  \eqref{step1} and \eqref{step2}, a Cauchy-Schwarz inequality gives
\begin{align*}
&\left|\sum_{K\in \mathcal{T}_N}
	\int_{\partial K}\frac{\partial^2 w}{\partial n^2}\frac{\partial \xi_N}{\partial n}\;\mathrm{d}s\right|  \\
&\le \sum_{K\in \mathcal{T}_N}|J_1+J_3|+\sum_{K\in \mathcal{T}_N}|J_2+J_4|\\
&\le 4
\left[\left(\sum_{K\in \mathcal{T}_N}h_x^4\left\|\frac{\partial^4 w}{\partial x\partial y^3}\right\|_{L^2(K)}^2\right)^{1/2}+\left(\sum_{K\in \mathcal{T}_N}h_y^4\left\|\frac{\partial^4 w}{\partial x^3\partial y}\right\|_{L^2(K)}^2\right)^{1/2}\right]
\left[\sum_{K\in \mathcal{T}_N}\|\nabla^2 \xi_N\|_{L^2(K)}^2\right]^{1/2}\\
&\le 4 \left[\left(\sum_{K\in \mathcal{T}_N}h_x^4\left\|\frac{\partial^4 w}{\partial x\partial y^3}\right\|_{L^2(K)}^2\right)^{1/2}+\left(\sum_{K\in \mathcal{T}_N}h_y^4\left\|\frac{\partial^4 w}{\partial x^3\partial y}\right\|_{L^2(K)}^2\right)^{1/2}\right] |\xi_N|_{2,N}.
\end{align*}
Thus, we have proved \eqref{consis1-1-eq}.

Applying the anisotropic inverse inequalities
\eqref{inverse1} to \eqref{step1} and \eqref{inverse2} to \eqref{step2} gives
\begin{align*}
\left|J_2+J_4\right|&\le C h_y \left\|\frac{\partial^4 w}{\partial x^3\partial y}\right\|_{L^2(K)} \left\|\frac{\partial \xi_N}{\partial x}\right\|_{L^2(K)},\\
\left|J_1+J_3\right|&\le C h_x \left\|\frac{\partial^4 w}{\partial x\partial y^3}\right\|_{L^2(K)} \left\|\frac{\partial \xi_N}{\partial y}\right\|_{L^2(K)}.
\end{align*}
Then, like the above derivation of~\eqref{consis1-1-eq}, one obtains  \eqref{consis1-2-eq}:
\begin{align*}
\left|\sum_{K\in \mathcal{T}_N}
	\int_{\partial K}\frac{\partial^2 w}{\partial n^2}\frac{\partial \xi_N}{\partial n}\;\mathrm{d}s\right|
&\le \sum_{K\in \mathcal{T}_N}|J_1+J_3|+\sum_{K\in \mathcal{T}_N}|J_2+J_4|\\
&\hspace{-2cm}\le C
\left[\sum_{K\in \mathcal{T}_N}\left(h_x^2\left\|\frac{\partial^4 w}{\partial x\partial y^3}\right\|_{L^2(K)}^2\right)^{1/2}
	+\sum_{K\in \mathcal{T}_N}\left(h_y^2\left\|\frac{\partial^4 w}{\partial x^3\partial y}\right\|_{L^2(K)}^2\right)^{1/2}\right]
		|\xi_N|_{1,N}.
\end{align*}
\end{proof}

\begin{remark}
At first sight \eqref{consis1-1-eq}, which bounds $\sum_{K\in \mathcal{T}_N}\int_{\partial K}
\frac{\partial^2 w}{\partial n^2}\frac{\partial \xi_N}{\partial n}\;\mathrm{d}s$ with factors  $h_x^2$ and $h_y^2$,
looks more attractive than \eqref{consis1-2-eq}, which yields  factors  $h_x$ and $h_y$.
But to continue the analysis with \eqref{consis1-1-eq} means working with $|\xi_N|_{2,N}$
and using the inequality $|\xi_N|_{2,N} \le \eps^{-1}\|\xi_N\|_{\eps,N}$, which introduces
an undesirable factor $\eps^{-1}$ into our error bounds.
Consequently we shall use~\eqref{consis1-2-eq} in our subsequent analysis, for then we have available
$|\xi_N|_{1,N} \le \|\xi_N\|_{\eps,N}$, which causes no difficulties.
\end{remark}

The next result was proved  in \cite[Lemma~3.3]{HuShiYang16} while assuming that
the mesh was quasi-uniform, but here we remove this restriction.

\begin{lemma}\label{consis2-1}
Assume that $w\in H_0^2(\Omega)\cap H^4(\Omega)$.
Then there exists a constant $C$ such that for all $\xi_N\in V_N$, one has
\begin{align}
&\left|\sum_{K\in \mathcal{T}_N}\int_{\partial K}
	\frac{\partial^2 w}{\partial s\partial n}\frac{\partial \xi_N}{\partial s}\,\mathrm{d}s\right| \notag\\
&\quad\le C \left[\left(\sum_{K\in \mathcal{T}_N}h_x^4\left\|\frac{\partial^4 w}{\partial x^2\partial y^2}\right\|_{L^2(K)}^2\right)^{1/2}
	+\left(\sum_{K\in \mathcal{T}_N}h_y^4\left\|\frac{\partial^4 w}{\partial x^2\partial y^2}\right\|_{L^2(K)}^2\right)^{1/2}\right]
		|\xi_N|_{2,N}  \label{consis2-1-eq}
\intertext{and}
&\left|\sum_{K\in \mathcal{T}_N}\int_{\partial K}
	\frac{\partial^2 w}{\partial s\partial n}\frac{\partial \xi_N}{\partial s}\,\mathrm{d}s\right| \notag\\
&\quad\le C \left[\sum_{K\in \mathcal{T}_N}\left(h_x^2\left\|\frac{\partial^4 w}{\partial x^2\partial y^2}\right\|_{L^2(K)}^2\right)^{1/2}+\sum_{K\in \mathcal{T}_N}\left(h_y^2\left\|\frac{\partial^4 w}{\partial x^2\partial y^2}\right\|_{L^2(K)}^2\right)^{1/2}\right] |\xi_N|_{1,N}. \label{consis2-2-eq}
\end{align}
\end{lemma}
\begin{proof}
The definition of $V_N$ implies that on each edge $e_i$ that is shared by two mesh rectangles
$K_1$ and $K_2$ (say), $\xi_N|_{e_i}$ is the same in both $K_1$ and $K_2$;
but on $e_i$ the derivative $\frac{\partial^2 w}{\partial s\partial n}$
will have opposite signs in these two rectangles, so
\begin{align}
\sum_{K\in \mathcal{T}_N}\int_{\partial K}\frac{\partial^2 w}{\partial s\partial n}\frac{\partial \xi_N}{\partial s}\;\mathrm{d}s
&=\sum_{K\in \mathcal{T}_N}\sum_{i=1}^4\int_{e_i}\frac{\partial^2 w}{\partial s\partial n}\frac{\partial \xi_N}{\partial s}\;\mathrm{d}s\notag\\
&=\sum_{K\in \mathcal{T}_N}\sum_{i=1}^4\int_{e_i}\frac{\partial^2 w}{\partial s\partial n}\mathcal{R}^0_{e_i}\frac{\partial \xi_N}{\partial s}\;\mathrm{d}s
=:\sum_{K\in \mathcal{T}_N}\sum_{i=1}^4J_i.
\label{Ji-2}
\end{align}

Similarly to the derivation of \eqref{J2J4a}, in the notation of Figure~\ref{fig:DOF} we have
\[  
J_2+J_4
=\int_{y_{c}-h_y}^{y_{c}+h_y}\left[\left(\frac{\partial^2 w}{\partial y\partial x}\mathcal{R}^0_{e_4}\frac{\partial \xi_N}{\partial y}\right)(x_c+h_x,y)-\left(\frac{\partial^2 w}{\partial y\partial x}\mathcal{R}^0_{e_2}\frac{\partial \xi_N}{\partial y}\right)(x_c-h_x,y)\right]\;\mathrm{d}y.
\]  
As $\xi_N\in V_N = \text{span} \{1,x,y,x^2,xy,y^2,x^3,y^3\}$,
it is easy to see that
\begin{equation*}
\left(\mathcal{R}^0_{e_2}\frac{\partial \xi_N}{\partial y}\right)(x_c-h_x,y)
	= \left(\mathcal{R}^0_{e_4}\frac{\partial \xi_N}{\partial y}\right)(x_c+h_x,y)
		\ \text{ for }y_{c}-h_y \le y \le y_{c}+h_y ,
\end{equation*}
and the definition of $\mathcal{R}^0_{e_2}$ implies that
$\int_{y_{c}-h_y}^{y_{c}+h_y}\left(\mathcal{R}^0_{e_2} \frac{\partial \xi_N}{\partial y}\right)(x_c-h_x,y)\;\mathrm{d}y=0$.
Hence
\begin{align*}
J_2+J_4
&= \int_{y_{c}-h_y}^{y_{c}+h_y}
	\left[\frac{\partial^2 w}{\partial y\partial x}(x_c+h_x,y) - \frac{\partial^2 w}{\partial y\partial x}(x_c-h_x,y) \right]
		\left(\mathcal{R}^0_{e_2}\frac{\partial \xi_N}{\partial y}\right)(x_c-h_x,y) \,\mathrm{d}y \notag\\
=  \int_{y_{c}-h_y}^{y_{c}+h_y}&
	\left[\left( \mathcal{R}^0_{e_4}\frac{\partial^2 w}{\partial y\partial x}\right)(x_c+h_x,y)
		- \left( \mathcal{R}^0_{e_2}\frac{\partial^2 w}{\partial y\partial x}\right)(x_c-h_x,y) \right]
		\left(\mathcal{R}^0_{e_2}\frac{\partial \xi_N}{\partial y}\right)(x_c-h_x,y) \,\mathrm{d}y.
\end{align*}

By a  Cauchy-Schwarz inequality we get
\begin{align}
(J_2+J_4)^2
&\le \left\{\int_{y_{c}-h_y}^{y_{c}+h_y}
	\left[\left( \mathcal{R}^0_{e_4}\frac{\partial^2 w}{\partial y\partial x}\right)(x_c+h_x,y)
		- \left( \mathcal{R}^0_{e_2}\frac{\partial^2 w}{\partial y\partial x}\right)(x_c-h_x,y) \right]^2 \,\mathrm{d}y\right\} \notag\\
&\qquad \times
	\int_{y_{c}-h_y}^{y_{c}+h_y} \left(\mathcal{R}^0_{e_2}\frac{\partial \xi_N}{\partial y}\right)^2(x_c-h_x,y) \,\mathrm{d}y.
	\label{J2J4-2}
\end{align}
Replacing $\partial^2 w/\partial x^2$ by $\partial^2 w/\partial y\partial x$
in the derivation of~\eqref{part2} yields
\begin{align*}
\int_{y_{c}-h_y}^{y_{c}+h_y}
	&\left[ \left( \mathcal{R}^0_{e_4}\frac{\partial^2 w}{\partial y\partial x}\right)(x_c+h_x,y)
		- \left( \mathcal{R}^0_{e_2}\frac{\partial^2 w}{\partial y\partial x}\right)(x_c-h_x,y) \right]^2
			\,\mathrm{d}y \notag\\
&\hspace{2cm}\le 8 h_x h_y^2 \left\|\frac{\partial^4 w}{\partial x^2\partial y^2}\right\|_{L^2(K)}^2,  
\end{align*}
while similarly to the derivation of \eqref{part1}, one has
\[ 
\int_{y_{c}-h_y}^{y_{c}+h_y} \left(\mathcal{R}^0_{e_2}\frac{\partial \xi_N}{\partial y}\right)^2(x_c-h_x,y) \,\mathrm{d}y
\le  \frac{2h_y^2}{h_x} \left\|\frac{\partial^2 \xi_N}{\partial y^2}\right\|_{L^2(K)}^2.
\] 
Hence \eqref{J2J4-2} implies that
\begin{equation}
\left|J_2+J_4\right| \le 4 h_y^2 \left\|\frac{\partial^4 w}{\partial x^2\partial y^2}\right\|_{L^2(K)}
	\left\|\frac{\partial^2 \xi_N}{\partial y^2}\right\|_{L^2(K)}.  \label{step1-2}
\end{equation}

One can show similarly that
\begin{equation}\label{step2-2}
\left|J_1+J_3\right| \le 4 h_x^2 \left\|\frac{\partial^4 w}{\partial x^2\partial y^2}\right\|_{L^2(K)}
	\left\|\frac{\partial^2 \xi_N}{\partial x^2}\right\|_{L^2(K)}.
\end{equation}

From \eqref{Ji-2}, \eqref{step1-2} and \eqref{step2-2} and  some Cauchy-Schwarz inequalities, we can derive
\begin{align*}
&\left|\sum_{K\in \mathcal{T}_N} \int_{\partial K}\frac{\partial^2 w}
	{\partial s\partial n}\frac{\partial \xi_N}{\partial s}\;\mathrm{d}s\right|
	\le \sum_{K\in \mathcal{T}_N}|J_1+J_3|+\sum_{K\in \mathcal{T}_N}|J_2+J_4|\\
&\le 4
\left[\left(\sum_{K\in \mathcal{T}_N}h_x^4\left\|\frac{\partial^4 w}{\partial x^2\partial y^2}\right\|_{L^2(K)}^2\right)^{1/2}+\left(\sum_{K\in \mathcal{T}_N}h_y^4\left\|\frac{\partial^4 w}{\partial x^2\partial y^2}\right\|_{L^2(K)}^2\right)^{1/2}\right]
\left[\sum_{K\in \mathcal{T}_N}\|\nabla^2 \xi_N\|_{L^2(K)}^2\right]^{1/2}\\
&\le 4 \left[\left(\sum_{K\in \mathcal{T}_N}h_x^4\left\|\frac{\partial^4 w}{\partial x^2\partial y^2}\right\|_{L^2(K)}^2\right)^{1/2}+\left(\sum_{K\in \mathcal{T}_N}h_y^4\left\|\frac{\partial^4 w}{\partial x^2\partial y^2}\right\|_{L^2(K)}^2\right)^{1/2}\right] |\xi_N|_{2,N},
\end{align*}
which proves \eqref{consis2-1-eq}.

Invoking the anisotropic inverse inequalities \eqref{inverse1}--\eqref{inverse2},
from \eqref{step1-2} and \eqref{step2-2} one has
\begin{align*}
\left|J_2+J_4\right|&\le C h_y \left\|\frac{\partial^4 w}{\partial x^2\partial y^2}\right\|_{L^2(K)} \left\|\frac{\partial \xi_N}{\partial y}\right\|_{L^2(K)},\\
\left|J_1+J_3\right|&\le C h_x \left\|\frac{\partial^4 w}{\partial x^2\partial y^2}\right\|_{L^2(K)} \left\|\frac{\partial \xi_N}{\partial x}\right\|_{L^2(K)}.
\end{align*}
Then one obtains  \eqref{consis2-2-eq} by imitating the above derivation of~\eqref{consis2-1-eq}.
\end{proof}

Lemmas~\ref{consis1-1} and~\ref{consis2-1} are anisotropic generalisations
on nonuniform meshes of the consistency error bounds
proved in~\cite{HuShiYang15,HuShiYang16} for discretisations of
a classical biharmonic problem
($\eps=1$ and $c\equiv 0$ in~\eqref{BVP}) on uniform meshes.

\section{Error analysis on a Shishkin mesh}\label{sec:Shishkin}

All our analysis so far is valid on the general tensor-product mesh defined at the start of Section~\ref{sec:Morley}, but to obtain good accuracy in our numerical solution of~\eqref{BVP}
one should use a mesh that is designed to handle boundary layers in singularly perturbed problems.
We now restrict our attention to a well-known tensor-product mesh of this type: the Shishkin mesh.

\subsection{The Shishkin mesh}\label{sec:ShishkinMesh}
Recall that $N$ is the number of mesh intervals in each coordinate direction.
Assume from now on that $N$~is divisible by~$4$. Define the transition parameter
\begin{equation*}
\lambda = \min\left\{\eps\ln N,  \, \frac{1}{4}\right\}.
\end{equation*}
We assume that $\lambda = \eps\ln N$ as otherwise $N$ would be exponentially large compared with $1/\eps$,
and then the problem can be solved accurately using a uniform mesh.
In practice $\eps$ is small, so it is extremely unlikely that $\eps \ln N > 1/4$.

Partition the interval $[0,1]$ by a piecewise uniform mesh that is obtained by
dividing $[0,\lambda]$ into $N/4$ equal subintervals,
$[\lambda, 1-\lambda]$ into $N/2$ equal subintervals and $[1-\lambda,1]$ into $N/4$ equal subintervals.
The tensor product of two such 1D meshes is our 2D Shishkin mesh for the problem~\eqref{BVP};
it is shown in Figure~\ref{fig:ShishRegions}.
(See \cite{RoosStynesTobiska08,StyStyBook} for further discussion of Shishkin meshes.)
We continue to write~$\mathcal{T}_N$ for the 2D mesh.
For any mesh sub-domain $\tilde{\Omega}$ of~$\Omega$,
denote the mesh on $\tilde{\Omega}$ by $\mathcal{T}_N(\tilde{\Omega})$.

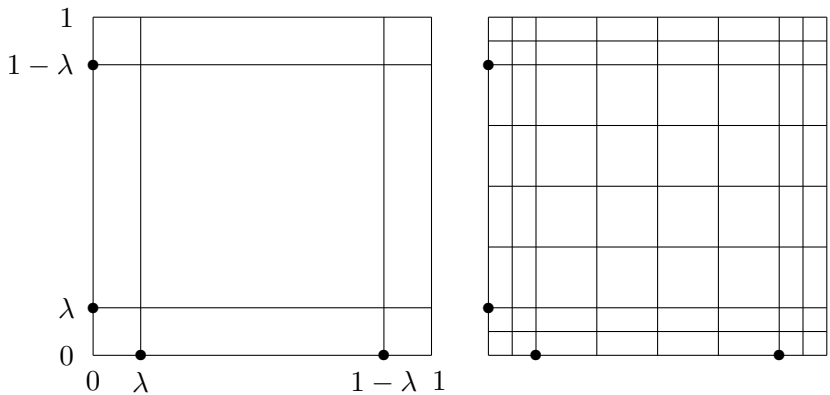
\begin{figure}[htbp]\unitlength0.5pt
\begin{center}
  \begin{picture}(256,256)
   \put(  0, 0){\line(0,1){256}}
   \put(256, 0){\line(0,1){256}}
   \put(220, 0){\line(0,1){256}}
   \put(36, 0){\line(0,1){256}}
   \put(  0, 0){\line(1,0){256}}
   \put(  0, 36){\line(1,0){256}}
   \put(  0,220){\line(1,0){256}}
   \put(  0,256){\line(1,0){256}}
   \put(0, -19){\makebox(0,0){$0$}}
   \put(36, -20){\makebox(0,0){$\lambda$}}
   \put(220, -20){\makebox(0,0){$1-\lambda$}}
   \put(262, -19){\makebox(0,0){$1$}}
   \put(-20,0){\makebox(0,0){$0$}}
   \put(-20,36){\makebox(0,0){$\lambda$}}
   \put(-40,220){\makebox(0,0){$1-\lambda$}}
   \put(-20,256){\makebox(0,0){$1$}}
   \put(36,0){\circle*{7}}
   \put(220,0){\circle*{7}}
   \put(0,220){\circle*{7}}
   \put(0,36){\circle*{7}}
  \end{picture}
  \hspace{5mm}
%
 \begin{picture}(256,256)
   \put(  0, 0){\line(0,1){256}}
   \put(18, 0){\line(0,1){256}}
   \put(36, 0){\line(0,1){256}}
   \put(82, 0){\line(0,1){256}}
   \put(128, 0){\line(0,1){256}}
   \put(174, 0){\line(0,1){256}}
   \put(220, 0){\line(0,1){256}}
   \put(238, 0){\line(0,1){256}}
   \put(256, 0){\line(0,1){256}}
   \put(  0, 0){\line(1,0){256}}
   \put(  0, 18){\line(1,0){256}}
   \put(  0, 36){\line(1,0){256}}
   \put(  0, 82){\line(1,0){256}}
   \put(  0, 128){\line(1,0){256}}
   \put(  0, 174){\line(1,0){256}}
   \put(  0,220){\line(1,0){256}}
   \put(  0,238){\line(1,0){256}}
   \put(  0,256){\line(1,0){256}}
   \put(36,0){\circle*{7}}
   \put(220,0){\circle*{7}}
   \put(0,220){\circle*{7}}
   \put(0,36){\circle*{7}}
   \end{picture}
\end{center}
\smallskip
\caption{A rectangular Shishkin mesh with $N=8$}
\label{fig:ShishRegions}
\end{figure}

In our error analysis, we use the Shishkin mesh subdomains $\Omega_1$,..., $\Omega_9$
that are displayed in Figure~\ref{fig:omega12}.
For convenience, define $\widehat{\Omega}_1:=\Omega_1\cup\Omega_4\cup\Omega_7$,
$\widehat{\Omega}_2:=\Omega_2\cup\Omega_5\cup\Omega_8$ and
$\widehat{\Omega}_3:=\Omega_3\cup\Omega_6\cup\Omega_9$.

\begin{figure}[htbp]\unitlength0.5pt
\unitlength1.0pt
\begin{center}
  \begin{picture}(128,128)
   \put(  0, 0){\line(0,1){128}}
   \put(128, 0){\line(0,1){128}}
   \put(  0, 0){\line(1,0){128}}
   \put(  0,128){\line(1,0){128}}
   \put(  18, 0){\line(0,1){128}}
   \put( 110, 0){\line(0,1){128}}
   \put(  0, 18){\line(1,0){128}}
   \put(  0,110){\line(1,0){128}}

   \put(-10,0){\makebox(0,0){$0$}}
   \put(-10,128){\makebox(0,0){$1$}}
   \put(0,-10){\makebox(0,0){$0$}}
   \put(128,-10){\makebox(0,0){$1$}}

   \put(18,-10){\makebox(0,0){$\lambda$}}
   \put(110,-10){\makebox(0,0){$1-\lambda$}}
   \put(-15,18){\makebox(0,0){$\lambda$}}
   \put(-15,110){\makebox(0,0){$1-\lambda$}}

   \put(18,0){\circle*{3}}
   \put(0,18){\circle*{3}}
   \put(110,0){\circle*{3}}
   \put(0,110){\circle*{3}}

   \put(9,9){\makebox(0,0){$\Omega_{1}$}}
   \put(64,9){\makebox(0,0){$\Omega_{2}$}}
   \put(119,9){\makebox(0,0){$\Omega_{3}$}}
   \put(9,64){\makebox(0,0){$\Omega_{4}$}}
   \put(64,64){\makebox(0,0){$\Omega_{5}$}}
   \put(119,64){\makebox(0,0){$\Omega_{6}$}}
   \put(9,119){\makebox(0,0){$\Omega_{7}$}}
   \put(64,119){\makebox(0,0){$\Omega_{8}$}}
   \put(119,119){\makebox(0,0){$\Omega_{9}$}}

  \end{picture}
  \end{center}
  \caption{Mesh subdomains in $\Omega$}
  \label{fig:omega12}
\end{figure}
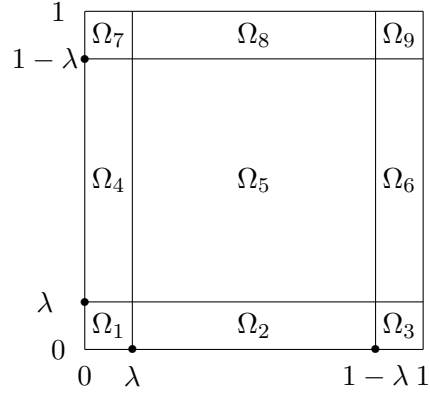

%
%
We shall now derive an error bound for the numerical solution $u_N$ of~\eqref{BVP_FEM}
by specialising our earlier results to the Shishkin mesh.
To do this we analyse the consistency error of the discretisation of each term in the PDE~\eqref{BVPeq} in several lemmas,
then combine these bounds in the proof of Theorem~\ref{thm:cgce}  to give the final convergence result.

\subsection{Consistency error of $a_N(\cdot,\cdot)$}\label{sec:erroraN}
\begin{lemma}\label{consis1-3}
Suppose that $u\in H_0^2(\Omega)\cap H^4(\Omega)$ satisfies Assumption \ref{ass:S}.
Then there exists a constant~$C$ such that for all $\xi_N\in V_N$, one has
\[
\eps^2\left|\sum_{K\in \mathcal{T}_N}\int_{\partial K}\frac{\partial^2 u}{\partial n^2}
	\frac{\partial \xi_N}{\partial n}\,\mathrm{d}s\right|
		\le C \left(\eps^{1/2}N^{-1}  + \eps N^{-1}\ln N  +N^{-2}\right) \|\xi_N\|_{\eps,N}.
\]
\end{lemma}
\begin{proof}
We apply the bound \eqref{consis1-2-eq} of Lemma~\ref{consis1-1} to each of the components of~$u$
that are provided by Assumption~\ref{ass:S}.

For the smooth component $S$, the bound \eqref{priS} and $h_x \le 2N^{-1}$ yield
\begin{equation}\label{consisS}
\eps^2\left[\sum_{K\in \mathcal{T}_N}h_x^2\left\|\frac{\partial^4 S}{\partial x\partial y^3}\right\|_{L^2(K)}^2\right]^{1/2}\le C \eps^2 N^{-1}.
\end{equation}

For the layer component $E_1$, using \eqref{priE1} on $\widehat{\Omega}_2$ we get
$$
\sum_{K\in \mathcal{T}_N(\widehat{\Omega}_2)}h_x^2\left\|\frac{\partial^4 E_1}{\partial x\partial y^3}\right\|_{L^2(K)}^2\le C N^{-2}\eps^{-4}\int_0^1\int_{\eps\ln N}^{1-\eps\ln N}e^{-2y/\eps}\,\mathrm{d}x\,\mathrm{d}y\le CN^{-2}\eps^{-4}\eps\le C\eps^{-3}N^{-2},
$$
and on $\widehat{\Omega}_1\cup\widehat{\Omega}_3$
\begin{align*}
\sum_{K\in \mathcal{T}_N(\widehat{\Omega}_1\cup\widehat{\Omega}_3)}
h_x^2\left\|\frac{\partial^4 E_1}{\partial x\partial y^3}\right\|_{L^2(K)}^2
&\le C (\eps N^{-1}\ln N)^2\eps^{-4}\int_0^1\left(\int_0^{\eps\ln N}+\int_{1-\eps\ln N}^1\right)e^{-2y/\eps}\,\mathrm{d}x\,\mathrm{d}y\\
&\le C(\eps N^{-1}\ln N)^2\eps^{-4}\eps(\eps\ln N)
	= CN^{-2}(\ln N)^3.
\end{align*}
Combining this pair of estimates and recalling that $\eps\ln N \le 1/4$, we obtain
\begin{equation}\label{consisE1}
\eps^2\left[\sum_{K\in \mathcal{T}_N}h_x^2\left\|\frac{\partial^4 E_1}{\partial x\partial y^3}\right\|_{L^2(K)}^2\right]^{1/2}\le C\eps^{1/2}N^{-1}.
\end{equation}
The estimates for $E_2$, $E_3$ and $E_4$ are similar.

For the layer component $E_{12}$, using \eqref{priE12} on $\widehat{\Omega}_2$ gives
\begin{align*}
\sum_{K\in \mathcal{T}_N(\widehat{\Omega}_2)}h_x^2\left\|\frac{\partial^4 E_{12}}{\partial x\partial y^3}\right\|_{L^2(K)}^2&\le C N^{-2}\eps^{-6}\int_0^1\int_{\eps\ln N}^{1-\eps\ln N}e^{-2x/\eps}e^{-2y/\eps}\,\mathrm{d}x\,\mathrm{d}y\\
&\le C\eps^{-4}N^{-4},
\end{align*}
and on $\widehat{\Omega}_1\cup\widehat{\Omega}_3$
\begin{align*}
\sum_{K\in \mathcal{T}_N(\widehat{\Omega}_1\cup\widehat{\Omega}_3)}
h_x^2\left\|\frac{\partial^4 E_{12}}{\partial x\partial y^3}\right\|_{L^2(K)}^2
&\le C (\eps N^{-1}\ln N)^2\eps^{-6}\int_0^1\left(\int_0^{\eps\ln N}+\int_{1-\eps\ln N}^1\right)e^{-2x/\eps}e^{-2y/\eps}\,\mathrm{d}x\,\mathrm{d}y\\
&\le C\eps^{-2}(N^{-1}\ln N)^2.
\end{align*}
Thus,
\begin{equation}\label{consisE12}
\eps^2\left[\sum_{K\in \mathcal{T}_N}h_x^2\left\|\frac{\partial^4 E_{12}}{\partial x\partial y^3}\right\|_{L^2(K)}^2\right]^{1/2}\le C(\eps N^{-1}\ln N + N^{-2}).
\end{equation}
The estimates for $E_{23}$, $E_{34}$ and $E_{41}$ are similar.

Adding \eqref{consisS}, \eqref{consisE1} and \eqref{consisE12}, we  get
\begin{equation}\label{consisu}
\eps^2\left[\sum_{K\in \mathcal{T}_N}h_x^2\left\|\frac{\partial^4 u}{\partial x\partial y^3}\right\|_{L^2(K)}^2\right]^{1/2}
	\le C(\eps^{1/2}N^{-1} + \eps N^{-1}\ln N +N^{-2}).
\end{equation}
Similarly,
\begin{equation*}
\eps^2\left[\sum_{K\in \mathcal{T}_N}h_y^2\left\|\frac{\partial^4 u}{\partial x^3\partial y}\right\|_{L^2(K)}^2\right]^{1/2}
	\le C\left(\eps^{1/2}N^{-1} + \eps N^{-1}\ln N +N^{-2}\right).
\end{equation*}
Now the bound \eqref{consis1-1-eq} of Lemma~\ref{consis1-1} and $|\xi_N|_{1,N} \le \|\xi_N\|_{\eps,N}$ yield
\begin{equation*}
\eps^2\left|\sum_{K\in \mathcal{T}_N}\int_{\partial K}\frac{\partial^2 u}{\partial n^2}\frac{\partial \xi_N}{\partial n}\,\mathrm{d}s\right|
	\le C \left(\eps^{1/2}N^{-1}  + \eps N^{-1}\ln N  +N^{-2}\right) \|\xi_N\|_{\eps,N}.
\end{equation*}
\end{proof}

\begin{lemma}\label{consis2-3}
Suppose that $u\in H_0^2(\Omega)\cap H^4(\Omega)$ satisfies Assumption \ref{ass:S}.
Then there exists a constant~$C$ such that for all $\xi_N\in V_N$, one has
\begin{align*}
\eps^2\left|\sum_{K\in \mathcal{T}_N}\int_{\partial K}\frac{\partial^2 u}{\partial s\partial n}\frac{\partial \xi_N}{\partial s}\;\mathrm{d}s\right|
\le C \left(\eps N^{-1}\ln N+N^{-2}\right)  \|\xi_N\|_{\eps,N}.
\end{align*}
\end{lemma}
\begin{proof}
Imitate the proof of Lemma~\ref{consis1-3} but invoke Lemma~\ref{consis2-1} instead of Lemma~\ref{consis1-1}.
To be specific, Lemma~\ref{consis2-1} gives bounds on the smooth and corner layer components
that are the same as \eqref{consisS} and \eqref{consisE12},
while for the boundary layer components, the bound $C \eps^{1/2}N^{-1}$ of~\eqref{consisE1}
is replaced by the smaller bound $C \eps^{3/2}N^{-1}$.
\end{proof}

\begin{remark}\label{compare}
[Comparison with Adini element]
For the highest-order term $\eps^2\Delta^2 u(x,y)$ of~\eqref{BVPeq},
the consistency error of the rectangular Morley element on the Shishkin mesh
is better than the corresponding bound for the rectangular Adini element, which was discussed in~\cite{MengStynes19}.
For in~\cite[eq.(4.27)]{MengStynes19},  the consistency error of the Adini element is bounded by
\begin{align*}
&\left[
\left(\sum_{K\in \mathcal{T}_N}h_x^4\left\|\frac{\partial^4 u}{\partial y^4}\right\|_{L^2(K)}^2\right)^{1/2}
+\left(\sum_{K\in \mathcal{T}_N}h_x^4\left\|\frac{\partial^4 u}{\partial x\partial y^3}\right\|_{L^2(K)}^2\right)^{1/2} \right.\\
&\left. +\left(\sum_{K\in \mathcal{T}_N}h_y^4\left\|\frac{\partial^4 u}{\partial x^4}\right\|_{L^2(K)}^2\right)^{1/2}
+\left(\sum_{K\in \mathcal{T}_N}h_y^4\left\|\frac{\partial^4 u}{\partial x^3\partial y}\right\|_{L^2(K)}^2\right)^{1/2}
\right] |\xi_N|_{2,N},
\end{align*}
or
\begin{align*}
&\left[
\left(\sum_{K\in \mathcal{T}_N}h_x^2\left\|\frac{\partial^4 u}{\partial y^4}\right\|_{L^2(K)}^2\right)^{1/2}
+\left(\sum_{K\in \mathcal{T}_N}h_x^2\left\|\frac{\partial^4 u}{\partial x\partial y^3}\right\|_{L^2(K)}^2\right)^{1/2} \right.\\
&\left. +\left(\sum_{K\in \mathcal{T}_N}h_y^2\left\|\frac{\partial^4 u}{\partial x^4}\right\|_{L^2(K)}^2\right)^{1/2}
+\left(\sum_{K\in \mathcal{T}_N}h_y^2\left\|\frac{\partial^4 u}{\partial x^3\partial y}\right\|_{L^2(K)}^2\right)^{1/2}
\right] |\xi_N|_{1,N}
\end{align*}
after using an inverse inequality.
In the above bounds, terms such as $h_x^2\left\|\frac{\partial^4 u}{\partial y^4}\right\|_{L^2(K)}$
or $h_x\left\|\frac{\partial^4 u}{\partial y^4}\right\|_{L^2(K)}$ are troublesome because of the mismatch
where derivatives in one variable are multiplied by mesh widths in the other variable---this is problematic for example
on mesh rectangles~$K\subset \Omega_2$ where $|\partial^4 u/\partial y^4| = O(\eps^{-3})$ by~\eqref{priE1}
but $h_x = O(N^{-1})$, leading to bounds involving negative powers of~$\eps$.
Remarkably, the rectangular Morley element does not have this drawback;
the estimates of Lemmas~\ref{consis1-1} and~\ref{consis2-1}  do not have similar mismatches
and consequently we are able to prove the error bounds of Lemmas~\ref{consis1-3} and~\ref{consis2-3}
in which no negative powers of~$\eps$ appear.
\end{remark}

Recall the piecewise bilinear interpolation operator $I_N$ of Section~\ref{sec:discretisation}
and the piecewise constant projector $P^0_K$ of Section~\ref{sec:prelim}.

\begin{lemma}\label{trun-1}
Suppose that $u\in H_0^2(\Omega)\cap H^4(\Omega)$ satisfies Assumption \ref{ass:S}.
Then there exists a constant~$C$ such that for all $\xi_N\in V_N$, one has
\begin{align*}
\eps^2\left|\sum_{K\in \mathcal{T}_N}\int_{K} \nabla\Delta u \nabla (\xi_N-I_N \xi_N)\,\mathrm{d}x\,\mathrm{d}y\right|
	\le C \left(\eps^{1/2}N^{-1}+\eps N^{-1}\ln N + N^{-2}\right) \|\xi_N\|_{\eps,N}.
\end{align*}
\end{lemma}
\begin{proof}
Let  $\xi_N\in V_N$ and $K\in \mathcal{T}_N$.
Similarly to the proof of \cite[Lemma 3.1]{HuShiYang15}, one sees that
\begin{equation*}
\int_{K} \frac{\partial(\xi_N-I_N \xi_N)}{\partial x}\;\,\mathrm{d}x\,\mathrm{d}y=\int_{K} \frac{\partial(\xi_N-I_N \xi_N)}{\partial y}\;\,\mathrm{d}x\,\mathrm{d}y=0
\end{equation*}
by verifying this property for each of the functions $1,x,y,x^2,xy,y^2,x^3,y^3$.
Hence
\begin{align*}
\left|\sum_{K\in \mathcal{T}_N}\int_{K} \nabla\Delta u \nabla (\xi_N-I_N \xi_N)\,\mathrm{d}x\,\mathrm{d}y\right|
&= \left|\sum_{K\in \mathcal{T}_N}\int_{K} \left(\nabla\Delta u-P_K^0\nabla\Delta u\right) \nabla(\xi_N-I_N \xi_N)
 	\,\mathrm{d}x\,\mathrm{d}y\right| \\
&\hspace{-25mm}\le C\left(\sum_{K\in\mathcal{T}_N}\left\|(\nabla\Delta u-P_K^0\nabla\Delta u)\right\|_{L^2(K)}^2\right)^{1/2}
\left(\sum_{K\in\mathcal{T}_N}\|\nabla \xi_N\|_{L^2(K)}^2\right)^{1/2},
\end{align*}
where we used  \eqref{H1stab} and Cauchy-Schwarz inequalities.
But the anisotropic  projection error bound \eqref{L2pro} yields
\begin{align*}
&\hspace{-3mm}\sum_{K\in\mathcal{T}_N}\|(\nabla\Delta u-P_K^0\nabla\Delta u)\|_{L^2(K)}^2 \\
&\le C\sum_{K\in \mathcal{T}_N}
\bigg(
h_x^2\left\|\frac{\partial^4 u}{\partial x^4}\right\|_{L^2(K)}^2+
h_y^2\left\|\frac{\partial^4 u}{\partial x^3\partial y}\right\|_{L^2(K)}^2+
h_x^2\left\|\frac{\partial^4 u}{\partial x^3\partial y}\right\|_{L^2(K)}^2+
h_y^2\left\|\frac{\partial^4 u}{\partial x^2\partial y^2}\right\|_{L^2(K)}^2\\
&\quad+h_x^2\left\|\frac{\partial^4 u}{\partial x^2\partial y^2}\right\|_{L^2(K)}^2+
h_y^2\left\|\frac{\partial^4 u}{\partial x\partial y^3}\right\|_{L^2(K)}^2+
h_x^2\left\|\frac{\partial^4 u}{\partial x\partial y^3}\right\|_{L^2(K)}^2+
h_y^2\left\|\frac{\partial^4 u}{\partial y^4}\right\|_{L^2(K)}^2
\bigg).
\end{align*}
Here the worst-behaved terms are $h_x^2\left\|\frac{\partial^4 u}{\partial x\partial y^3}\right\|_{L^2(K)}^2$ and $h_y^2\left\|\frac{\partial^4 u}{\partial x^3\partial y}\right\|_{L^2(K)}^2$.
These terms can be bounded like \eqref{consisu}, obtaining
\begin{equation*}
\eps^2\left(\sum_{K\in\mathcal{T}_N}\|(\nabla\Delta u-P_K^0\nabla\Delta u)\|_{L^2(K)}^2\right)^{1/2}
	\le C \left(\eps^{1/2}N^{-1}+\eps N^{-1}\ln N + N^{-2}\right).
\end{equation*}
The lemma now follows since $\left(\sum_{K\in\mathcal{T}_N}\|\nabla \xi_N\|_{L^2(K)}^2\right)^{1/2} \le  \|\xi_N\|_{\eps,N}$.
\end{proof}

\subsection{Consistency error of $b_N(\cdot,\cdot)$}\label{sec:errorbN}
In this subsection we consider the consistency error of the lower-order term in the PDE~\eqref{BVPeq}.
The components $u^E$ and $S$ of $u$ that were defined in Assumption~\ref{ass:S}
are handled separately.

\begin{lemma}\label{trun-4E}
There exists a constant~$C$ such that for all $\xi_N\in V_N$, one has
\[
\left|b_N\left(u^E,\xi_N - I_N \xi_N\right)\right|
\le C \left(\eps^{1/2}N^{-1}+\eps N^{-1}\ln N + N^{-2}\right) \|\xi_N\|_{\eps,N}.
\]
\end{lemma}
\begin{proof}
Recall that $u^E(x,y) = \sum_{i = 1}^4 E_i(x,y) + E_{12}(x,y) + E_{23}(x,y) + E_{34}(x,y) + E_{14}(x,y)$. Similarly to the proof of Lemma \ref{trun-1}, one has
\begin{align*}
\left|b_N(u^E,\xi_N - I_N\xi_N) \right| &=
\left|\sum_{K\in \mathcal{T}_N}\int_{K} c \nabla u^E\nabla (\xi_N-I_N \xi_N)\,\mathrm{d}x\,\mathrm{d}y\right| \\
&\le\left|\sum_{K\in \mathcal{T}_N}\int_{K} P_K^0 c\left(\nabla u^E-P_K^0\nabla u^E\right)\nabla (\xi_N-I_N \xi_N)
	\,\mathrm{d}x\,\mathrm{d}y \right| \\
&\qquad+\left|\sum_{K\in \mathcal{T}_N}\int_{K} \left(c-P_K^0 c\right)\nabla u^E\nabla (\xi_N-I_N \xi_N)\,\mathrm{d}x\,\mathrm{d}y \right| \\
&\hspace{-28mm}\le C\left(\sum_{K\in\mathcal{T}_N} \left\|\left(\nabla u^E-P_K^0\nabla u^E\right)\right\|_{L^2(K)}^2\right)^{1/2}
\left(\sum_{K\in\mathcal{T}_N}\|\nabla \xi_N\|_{L^2(K)}^2\right)^{1/2}\\
&\hspace{-20mm} + C\left(\sum_{K\in\mathcal{T}_N} \left\|\left(\nabla c-P_K^0\nabla c\right)\right\|_{L^2(K)}^2\right)^{1/2}
\left(\sum_{K\in\mathcal{T}_N}\|\nabla \xi_N\|_{L^2(K)}^2\right)^{1/2} \\
&\hspace{-28mm}\le C\left[\left(\sum_{K\in\mathcal{T}_N} \left\|\left(\nabla u^E-P_K^0\nabla u^E\right)\right\|_{L^2(K)}^2\right)^{1/2}
	+ \left(\sum_{K\in\mathcal{T}_N} \left\|\left(\nabla c-P_K^0\nabla c\right)\right\|_{L^2(K)}^2\right)^{1/2}\right]
	 \|\xi_N\|_{\eps,N},
\end{align*}
as $\left|\nabla u^E\right|\le C$ on $\Omega$ by Assumption \ref{ass:S}
and $\left(\sum_{K\in\mathcal{T}_N}\|\nabla \xi_N\|_{L^2(K)}^2\right)^{1/2} \le  \|\xi_N\|_{\eps,N}$.
These two terms have a similar structure, but
$c$ is well behaved so we give a detailed analysis only for the term containing~$u^E$.

The anisotropic projection error estimate \eqref{L2pro} yields
\begin{align*}
&\sum_{K\in\mathcal{T}_N} \left\|\left(\nabla u^E-P_K^0\nabla u^E\right)\right\|_{L^2(K)}^2 \\
&\le C\sum_{K\in \mathcal{T}_N}
\left(
h_x^2\left\|\frac{\partial^2 u^E}{\partial x^2}\right\|_{L^2(K)}^2+
h_y^2\left\|\frac{\partial^2 u^E}{\partial x\partial y}\right\|_{L^2(K)}^2+
h_x^2\left\|\frac{\partial^2 u^E}{\partial x\partial y}\right\|_{L^2(K)}^2+
h_y^2\left\|\frac{\partial^2 u^E}{\partial y^2}\right\|_{L^2(K)}^2\right).
\end{align*}
The boundary layer and corner layer components of $u^E$
are handled as in the proof of Lemma~\ref{consis1-3}.
The worst-behaved terms are
$h_x^2\left\|\frac{\partial^2 u^E}{\partial x\partial y}\right\|_{L^2(K)}^2$
and $h_y^2\left\|\frac{\partial^2 u^E}{\partial x\partial y}\right\|_{L^2(K)}^2$,
whereas in  Lemma~\ref{consis1-3} one has
$h_x^2\left\|\frac{\partial^4 u^E}{\partial x\partial y^3}\right\|_{L^2(K)}^2$; that is,
here there are two fewer orders of derivative but this is balanced by the additional factor
$\eps^2$ that is present in  Lemma~\ref{consis1-3}.
Thus, like \eqref{consisE1}, we get finally
\begin{align*}
&\left[\sum_{K\in\mathcal{T}_N}\bigg(h_x^2\left\|\frac{\partial^2 E_1}{\partial x^2} \right\|_{L^2(K)}^2
	+h_y^2\left\|\frac{\partial^2 E_1}{\partial x\partial y} \right\|_{L^2(K)}^2
	+h_x^2\left\|\frac{\partial^2 E_1}{\partial x\partial y} \right\|_{L^2(K)}^2
	+h_y^2\left\|\frac{\partial^2 E_1}{\partial y^2} \right\|_{L^2(K)}^2\bigg)\right]^{1/2}\\
&\hspace{2cm} \le C \eps^{1/2}N^{-1},
\end{align*}
and like \eqref{consisE12} we get
\begin{align*}
&\left[\sum_{K\in\mathcal{T}_N}\bigg(h_x^2\left\|\frac{\partial^2 E_{12}}{\partial x^2} \right\|_{L^2(K)}^2
	+h_y^2\left\|\frac{\partial^2 E_{12}}{\partial x\partial y} \right\|_{L^2(K)}^2
	+h_x^2\left\|\frac{\partial^2 E_{12}}{\partial x\partial y} \right\|_{L^2(K)}^2
	+h_y^2\left\|\frac{\partial^2 E_{12}}{\partial y^2} \right\|_{L^2(K)}^2\bigg)\right]^{1/2}\\
&\hspace{2cm}\le C (\eps N^{-1}\ln N + N^{-2}).
\end{align*}
These estimates yield
\begin{align}\label{similar1}
&\left[\sum_{K\in\mathcal{T}_N}\bigg(h_x^2\left\|\frac{\partial^2 u^E}{\partial x^2} \right\|_{L^2(K)}^2
	+h_y^2\left\|\frac{\partial^2 u^E}{\partial x\partial y} \right\|_{L^2(K)}^2
	+h_x^2\left\|\frac{\partial^2 u^E}{\partial x\partial y} \right\|_{L^2(K)}^2
	+h_y^2\left\|\frac{\partial^2 u^E}{\partial y^2} \right\|_{L^2(K)}^2\bigg)\right]^{1/2}\notag\\
&\hspace{2cm}\le C \left(\eps^{1/2}N^{-1}+\eps N^{-1}\ln N + N^{-2}\right).
\end{align}

Combining these bounds, we are done.
\end{proof}

It now remains only to bound the smooth component term $\left|b_N(S,\xi_N - I_N \xi_N)\right|$,
but obtaining a sharp estimate for this term turns out to be troublesome.

For the rectangular Shishkin mesh $\mathcal{T}_N$,
let $e$ be an edge of some mesh element~$K$.
If $e\not\subset\partial\Omega$, then there are two elements $K_L$ and $K_R$ in~$\mathcal{T}_N$ that share $e$ as a common edge.
Let $\omega_e:=K_L\cup K_R$ denote the patch associated with $e$.
Let $|K_L|$ and $|K_R|$ denote the areas of $K_L$ and $K_R$ respectively.
If $|K_L|=|K_R|$, then we call the patch uniform.

Let $\mathcal{E}_N=\mathcal{E}_N^1\bigcup\mathcal{E}_N^2\bigcup\mathcal{E}_N^3$
be the set of all edges of elements in~$\mathcal{T}_N$, where $\mathcal{E}_N^1$ contains the internal edges with uniform patches,
$\mathcal{E}_N^2$ contains the internal edges with nonuniform patches,
and $\mathcal{E}_N^3$ contains the edges in~$\partial\Omega$.

Recall the properties~\eqref{nodeedge} of the Morley basis.
Let $\mathcal{N}_N := \{(x_i,y_j): i,j=0,\dots,N\}$ denote the set of nodes of~$\mathcal{T}_N$.
For each $v_N\in V_N$, we shall treat separately the components of $v_N$ that correspond to nodes or edges.
Thus, set
$V_N^\mathcal{X}:= \left\{v_N\in V_N:\int_e\frac{\partial v_N}{\partial \mathbf{n}_e}ds=0\ \forall\,e\in\mathcal{E}_N \right\}$
and $V_N^\mathcal{E}:= \left\{v_N\in V_N: v_N(x)=0\ \forall\,x\in\mathcal{N}_N \right\}$.

We decompose $V_N^\mathcal{E}$ further. For each $e\in\mathcal{E}_N$, define
$$
V_N^e:= \left\{v_N\in V_N^\mathcal{E}: \int_{e'}\frac{\partial v_N}{\partial \mathbf{n}_{e'}}ds=0
	\ \forall\,e'\in\mathcal{E}_N\text{ with } e'\ne e \right\}.
$$

Set $P^\mathcal{X}(K)={\rm span}\{p_i,1\leqslant i\leqslant 4\}$, where the $p_i$ were defined in Section~\ref{sec:Morley}.
For each mesh element $K$ and each edge $e$ of $K$, set
\[
P^e(K):= \left\{v_N\in P(K): v_N = 0\text{ at each node of }K,
\int_{e'}\frac{\partial v_N}{\partial \mathbf{n}_{e'}}ds = 0 \ \forall  e'\subset\partial K \text{ with }e'\ne e \right\}.
\]
If $e$ is an internal edge with patch $\omega_e = K_L \cup K_R$, set
$$
P^e(\omega_e):= \left\{v_N\in L^2(\Omega):v_N|_K\in P^e(K)\text{ for }K= K_L,K_R
	\text{ and } \int_e\left[\frac{\partial v_N}{\partial{\mathbf{n}_{e}}}\right]\,ds =0  \right\}.
$$

The next two lemmas establish special properties of the rectangular Morley element.

\begin{lemma}\label{lem:super_point}
Let $K\in{\mathcal T}_N$ be a mesh element.
Let $\phi\in P^\mathcal{X}(K)$ and $q_K^1\in Q_1(K)$.
Then
\begin{equation}\label{super_point}
\int_{K} q_K^1 \frac{\partial(\phi-I_N \phi)}{\partial x}\,\mathrm{d}x\,\mathrm{d}y
	= \int_{K} q_K^1 \frac{\partial(\phi-I_N \phi)}{\partial y}\,\mathrm{d}x\,\mathrm{d}y = 0.
\end{equation}
\end{lemma}
\begin{proof}
Recall that $(x_{c},y_{c})$ is the center of $K$ and $K$ has dimensions $2h_x \times 2h_y$.
To simplify the presentation, set $\xi_x = (x - x_{c})/h_x$ and $\xi_y = (y - y_{c})/h_y$.
One can deduce from the definitions of the $p_i$ in Section~\ref{sec:Morley}
that for any $\phi\in P^\mathcal{X}(K)$, one has
$$
\phi-I_N\phi\in \textrm{span}\left\{\xi_x(\xi_x^2-1), \xi_y(\xi_y^2-1)\right\}.
$$
Hence
$$
\frac{\partial (\phi-I_N\phi)}{\partial x}\in \textrm{span}\left\{\frac{\partial [\xi_x(\xi_x^2-1)]}{\partial x}\right\}
	= \textrm{span}\left\{\frac{1}{h_x}(3\xi_x^2-1)\right\}
$$
and similarly for $\partial (\phi-I_N\phi)/\partial y$.
But
$$
\int_{-1}^1(3\xi_x^2-1)\,\mathrm{d} \xi_x = \int_{-1}^1 \xi_x (3\xi_x^2-1)\,\mathrm{d} \xi_x = 0,
$$
so $\int_{K} q_K^1 \frac{\partial(\phi-I_N \phi)}{\partial x}\,\mathrm{d}x\,\mathrm{d}y  = 0$ follows.
The other identity in~\eqref{super_point} is proved similarly.
\end{proof}

Let $Q_1(\omega_e)$ denote the space of polynomials of the form $c_0+c_1x+c_2y+c_3xy$ defined on~$\omega_e$
and let $P_0(\omega_e)$ denote the space of constant functions defined on~$\omega_e$.

\begin{lemma}
\label{super_edge}
Let $e\in\mathcal{E}_N$ be an edge.
\begin{enumerate}
\item[(i)]
If  $e\in\mathcal{E}_N^1\bigcup\mathcal{E}_N^2$, its associated patch $\omega_e$ is uniform,
 $\psi\in P^{e}(\omega_e)$ and $q_{\omega_e}^1\in Q_1(\omega_e)$, then
\begin{equation}
\label{eq:super_edge1}
\int_{\omega_e} q_{\omega_e}^1 \frac{\partial(\psi-I_N \psi)}{\partial x}\,\mathrm{d}x\,\mathrm{d}y
	= \int_{\omega_e} q_{\omega_e}^1 \frac{\partial(\psi-I_N \psi)}{\partial y}\,\mathrm{d}x\,\mathrm{d}y = 0.
\end{equation}
\item[(ii)]
If  $e\in\mathcal{E}_N^1\bigcup\mathcal{E}_N^2$, its associated patch $\omega_e$ is nonuniform,
 $\psi\in P^{e}(\omega_e)$ and $p_{\omega_e}^0\in P_0(\omega_e)$, then
\begin{equation}
\label{eq:super_edge2}
\int_{\omega_e} p_{\omega_e}^0 \frac{\partial(\psi-I_N \psi)}{\partial x}\,\mathrm{d}x\,\mathrm{d}y
	= \int_{\omega_e} p_{\omega_e}^0 \frac{\partial(\psi-I_N \psi)}{\partial y}\,\mathrm{d}x\,\mathrm{d}y = 0.
\end{equation}
\item[(iii)]
If $e\in\mathcal{E}_N^3$ with $e\subset \partial K$,  $\psi\in P^{e}(K)$ and $p_K^0\in P_0(K)$, then
\begin{equation}
\label{eq:super_edge3}
\int_{K} p_K^0 \frac{\partial(\psi-I_N \psi)}{\partial x}\,\mathrm{d}x\,\mathrm{d}y
	= \int_{K} p_K^0 \frac{\partial(\psi-I_N \psi)}{\partial y}\,\mathrm{d}x\,\mathrm{d}y = 0.
\end{equation}
\end{enumerate}
\end{lemma}

\begin{proof}
Assume that the edge $e$ is parallel to the $y$-coordinate axis.
(The case where $e$ is parallel to the $x$-coordinate axis is similar.)
Let the center of $e$ be the point $(e_{x,c},e_{y,c})$
and let the half-lengths of $K_L$ and $K_R$ in the $x$ direction be $h_x^L$ and $h_x^R$ respectively.
Then
$e = \{(e_{x,c},y): y=e_{y,c}+\xi_y h_y,\;-1\leq \xi_y\leq 1\}$
and
\begin{align*}
K_L &= \{(x,y): x = e_{x,c}+\xi_x h_x^L\text{ for }-2\leq \xi_x\leq 0, \ y=e_{y,c}+\xi_y h_y \text{ for }-1\leq \xi_y\leq 1\}, \\
K_R &= \{(x,y): x = e_{x,c}+\xi_x h_x^R\text{ for }0\leq \xi_x\leq 2, \ y=e_{y,c}+\xi_y h_y\text{ for }-1\leq \xi_y\leq 1\}.
\end{align*}
On $\omega_e$, take increasing $x$ as the positive direction.
Then (see Section~\ref{sec:Morley}) the basis function $\psi^{\omega_e}$ associated with $\omega_e$
is like $q_4$ in $K_L$ and $-q_2$ in $K_R$:
$$
\psi^{\omega_e} =\left\{\begin{array}{ll} \displaystyle
\frac{1}{4}
\left[\frac{x-(e_{x,c}-h_x^L)}{h_x^L}+1\right]^2
\left[\frac{x-(e_{x,c}-h_x^L)}{h_x^L}-1\right] & \mbox{on}\ K_L,
\\
\\
\displaystyle
\frac{1}{4}
\left[\frac{x-(e_{x,c}+h_x^R)}{h_x^R}+1\right]
\left[\frac{x-(e_{x,c}+h_x^R)}{h_x^R}-1\right]^2 & \mbox{on}\ K_R.
\end{array}
\right.
$$
Clearly $\psi^{\omega_e} = 0$ on the right and left edges of $K_L$ and of $K_R$.
Hence $I_N\psi^{\omega_e} \equiv 0$ and one gets
\begin{equation*}
\int_{K_L} p_{K_L}^0 \frac{\partial\left(\psi^{\omega_e}-I_N\psi^{\omega_e}\right)}{\partial x}\,\mathrm{d}x\,\mathrm{d}y
= \int_{K_R} p_{K_R}^0 \frac{\partial\left(\psi^{\omega_e}-I_N\psi^{\omega_e}\right)}{\partial x}\,\mathrm{d}x\,\mathrm{d}y = 0
\end{equation*}
and
\begin{equation*}
\int_{K_L} p_{K_L}^0 \frac{\partial\left(\psi^{\omega_e}-I_N\psi^{\omega_e}\right)}{\partial y}\,\mathrm{d}x\,\mathrm{d}y
= \int_{K_R} p_{K_R}^0 \frac{\partial\left(\psi^{\omega_e}-I_N\psi^{\omega_e}\right)}{\partial y}\,\mathrm{d}x\,\mathrm{d}y = 0,
\end{equation*}
from which \eqref{eq:super_edge2} and \eqref{eq:super_edge3} follow.

One also has $\partial\psi^{\omega_e}/\partial y \equiv 0$, so
$\int_{\omega_e} q_{\omega_e}^1 \frac{\partial\left(\psi^{\omega_e}-I_N\psi^{\omega_e}\right)}{\partial y}
	\,\mathrm{d}x \, \mathrm{d}y = 0$.
Moreover, if $h_x^L = h_x^R$ then from
\begin{equation*}
\int_{-1}^1 \frac{\partial [(\xi_x+1)^2(\xi_x-1)]}{\partial \xi_x} \xi_x \,\mathrm{d}\xi_x
	+ \int_{-1}^1 \frac{\partial [(\xi_x+1)(\xi_x-1)^2]}{\partial \xi_x} \xi_x \,\mathrm{d}\xi_x = 0
\end{equation*}
we get
$\int_{\omega_e} q_{\omega_e}^1 \frac{\partial\left(\psi^{\omega_e}-I_N\psi^{\omega_e}\right)}{\partial x}
	\,\mathrm{d}x\,\mathrm{d}y = 0$,
which yields~\eqref{eq:super_edge1}.
\end{proof}

Let $I_{\omega_e}: C(\omega_e)\rightarrow Q_1(\omega_e)$ denote
the bilinear nodal interpolation operator
on the patch~$\omega_e$ that, given $v\in C(\omega_e)$,  satisfies $I_{\omega_e} v(P)=v(P)$
for each of the four vertices $P$ of~$\omega_e$ that do not lie on~$e$.
Let $P_{\omega_e}^0: L^2(\omega_e)\rightarrow P_0(\omega_e)$ be the $L^2$ projector on patch $\omega_e$, viz.,
\begin{equation*}
(v,w_N)_{\omega_e} = (P_{\omega_e}^0 v,w_N)_{\omega_e}\quad \forall v\in L^2(\omega_e), \, \forall w_N\in P_0(\omega_e).
\end{equation*}

\begin{lemma}\label{trun-4S}
There exists a constant~$C$ such that for all $\xi_N\in V_N$, one has
\[
\left|b_N(S,\xi_N - I_N \xi_N)\right|\le C N^{-3/2} \|\xi_N\|_{\eps,N}.
\]
\end{lemma}
\begin{proof}
For each $\xi_N\in V_N$,  by \cite[Lemma 3.6]{MengYangZhangSCM16}
there exist $\xi_N^\mathcal{X}\in V_N^\mathcal{X}$ and $\xi_N^e\in V_N^e \,\forall e\in\mathcal{E}_N$ such that
$\xi_N=\xi_N^\mathcal{X}+\sum_{e\in\mathcal{E}_N} \xi_N^e $ and
\begin{equation}\label{basis_bound}
|\xi_N^\mathcal{X}|_{1,N}^2+\sum_{e\in\mathcal{E}_h}|\xi_N^e|_{1,N}^2
\le C|\xi_N|_{1,N}^2.
\end{equation}
Consequently
\begin{align}
b_N(S,\xi_N - I_N\xi_N)
&=\sum_{K\in \mathcal{T}_N}\int_{K} c \nabla S\,\nabla (\xi_N-I_N \xi_N)\,\mathrm{d}x\,\mathrm{d}y\notag\\
&=\sum_{K\in \mathcal{T}_N}\int_{K} c \nabla S\,\nabla (\xi_N^\mathcal{X}-I_N\xi_N^\mathcal{X})\,\mathrm{d}x\,\mathrm{d}y\notag\\
&\quad + \sum_{e\in\mathcal{E}_N}\sum_{K\in \mathcal{T}_N}\int_{K} c \nabla S\,\nabla (\xi_N^e-I_N \xi_N^e)\,\mathrm{d}x\,\mathrm{d}y.\label{sep_FEM}
\end{align}
We will deal with the right-hand side terms one by one.
First, using \eqref{super_point}, Cauchy-Schwarz inequalities and \eqref{priS} one has
\begin{align}
\left|\sum_{K\in \mathcal{T}_N}\int_{K} c \nabla S \nabla (\xi_N^\mathcal{X}-I_N \xi_N^\mathcal{X})\,\mathrm{d}x\,\mathrm{d}y\right|
&\le\left|\sum_{K\in \mathcal{T}_N}\int_{K}\left[c \nabla S-I_N\left(c \nabla S\right)\right]\nabla (\xi_N^\mathcal{X}-I_N \xi_N^\mathcal{X})\,\mathrm{d}x\,\mathrm{d}y \right| \notag\\
&\hspace{-20mm}\le C\left(\sum_{K\in\mathcal{T}_N} \left\|c \nabla S-I_N\left(c \nabla S\right)\right\|_{L^2(K)}^2\right)^{1/2}
\left(\sum_{K\in\mathcal{T}_N}\|\nabla \xi_N^\mathcal{X}\|_{L^2(K)}^2\right)^{1/2}\notag\\
&\hspace{-20mm}\le C N^{-2}|\xi_N^\mathcal{X}|_{1,N}.\label{casenode}
\end{align}
For the terms in $ \sum_{e\in\mathcal{E}_N}\dots$ in~\eqref{sep_FEM}, consider the three cases $e\in\mathcal{E}_N^1$,
$e\in\mathcal{E}_N^2$ and $e\in\mathcal{E}_N^3$.
The definition of $V_N$ implies that $\xi_N^e=0$ if $e\in\mathcal{E}_N^3$, so we need only consider
$e\in\mathcal{E}_N^1$ and $e\in\mathcal{E}_N^2$.

Using \eqref{eq:super_edge1}, Cauchy-Schwarz inequalities and \eqref{priS} we get
\begin{align}
\left|\sum_{e\in\mathcal{E}_N^1}\int_{\omega_e} c \nabla S \nabla (\xi_N^e-I_N \xi_N^e)\,\mathrm{d}x\,\mathrm{d}y\right|
&\le\left|\sum_{e\in\mathcal{E}_N^1}\int_{\omega_e}\left[c \nabla S-I_{\omega_e}\left(c \nabla S\right)\right]\nabla (\xi_N^e-I_N \xi_N^e)\,\mathrm{d}x\,\mathrm{d}y \right| \notag\\
&\hspace{-15mm}\le C \left(\sum_{e\in\mathcal{E}_N^1} \left\|c \nabla S-I_{\omega_e}\left(c \nabla S\right)\right\|_{L^2(\omega_e)}^2\right)^{1/2}
\left(\sum_{e\in\mathcal{E}_N^1}|\xi_N^e|_{1,N}^2\right)^{1/2}\notag\\
&\hspace{-15mm}\le C N^{-2}
\left(\sum_{e\in\mathcal{E}_N^1}|\xi_N^e|_{1,N}^2\right)^{1/2}.\label{caseedge1}
\end{align}

Using \eqref{eq:super_edge2}, \eqref{priS}, \eqref{L2pro} and \eqref{L2inter}, we obtain
\begin{align}
\left|\sum_{e\in\mathcal{E}_N^2}\int_{\omega_e} c \nabla S \nabla (\xi_N^e-I_N \xi_N^e)\,\mathrm{d}x\,\mathrm{d}y\right|
&\le\left|\sum_{e\in\mathcal{E}_N^2}\int_{\omega_e}\left[c \nabla S-P_{\omega_e}^0\left(c \nabla S\right)\right]\nabla (\xi_N^e-I_N \xi_N^e)\,\mathrm{d}x\,\mathrm{d}y \right| \notag\\
&\hspace{-15mm}\le C \left(\sum_{e\in\mathcal{E}_N^2} \left\|c \nabla S-P_{\omega_e}^0\left(c \nabla S\right)\right\|_{L^2(\omega_e)}^2\right)^{1/2}
\left(\sum_{e\in\mathcal{E}_N^2}|\xi_N^e|_{1,N}^2\right)^{1/2}\notag\\
&\hspace{-15mm}\le C N^{-1} \left(\sum_{e\in\mathcal{E}_N^2} \left|c \nabla S\right|_{H^1(\omega_e)}^2\right)^{1/2}
\left(\sum_{e\in\mathcal{E}_N^2}|\xi_N^e|_{1,N}^2\right)^{1/2}\notag\\
&\hspace{-15mm}\le C N^{-3/2}\left(\sum_{e\in\mathcal{E}_N^2}|\xi_N^e|_{1,N}^2\right)^{1/2},\label{caseedge2}
\end{align}
as  $\text{meas}(\omega_e) = O(N^{-2})$ and there are $O(N)$ edges in $\mathcal{E}_N^2$.

Combining \eqref{sep_FEM}--\eqref{caseedge2} then recalling \eqref{basis_bound} gives
$\left|b_N(S,\xi_N - I_N \xi_N)\right|\le C N^{-3/2} \|\xi_N\|_{\eps,N}$.
\end{proof}

\begin{lemma}\label{lem:ftrunc}
There exists a constant~$C$ such that for all $\xi_N\in V_N$ one has
\[
\left|\left(f,\xi_N - I_N \xi_N\right)\right|\le C N^{-3/2} \|\xi_N\|_{\eps,N}.
\]
\end{lemma}
\begin{proof}
Using the decomposition of $\xi_N\in V_N$ in~\eqref{basis_bound}, we have
\begin{equation}\label{sep_FEM_f}
(f,\xi_N - I_N\xi_N) = \sum_{K\in \mathcal{T}_N}\int_{K} f (\xi_N^\mathcal{X}-I_N\xi_N^\mathcal{X})\,\mathrm{d}x\,\mathrm{d}y
	+ \sum_{e\in\mathcal{E}_N}\sum_{K\in \mathcal{T}_N}\int_{K} f (\xi_N^e-I_N \xi_N^e)\,\mathrm{d}x\,\mathrm{d}y.
\end{equation}
We will deal with these  terms one by one.

Integrating by parts in \eqref{super_point}, one sees that on each element $K\in{\mathcal T}_N$,
for every $\phi\in P^\mathcal{X}(K)$ and every $p_K^0\in P_0(K)$ one has
$\int_{K} p_K^0 (\phi-I_N \phi)\,\mathrm{d}x\,\mathrm{d}y = 0$.
Hence we have
\begin{align}
\left|\sum_{K\in \mathcal{T}_N}\int_{K} f(\xi_N^\mathcal{X}-I_N \xi_N^\mathcal{X})\,\mathrm{d}x\,\mathrm{d}y\right|
&= \left|\sum_{K\in \mathcal{T}_N}\int_{K}\left(f-P_K^0f\right) (\xi_N^\mathcal{X}-I_N \xi_N^\mathcal{X})\,\mathrm{d}x\,\mathrm{d}y \right| \notag\\
&\hspace{-15mm}\le C\left(\sum_{K\in\mathcal{T}_N} \left\|f-P_K^0f\right\|_{L^2(K)}^2\right)^{1/2}
\left(\sum_{K\in\mathcal{T}_N}\|\xi_N^\mathcal{X}-I_N \xi_N^\mathcal{X}\|_{L^2(K)}^2\right)^{1/2}\notag\\
&\hspace{-15mm}\le C N^{-2}|\xi_N^\mathcal{X}|_{1,N}, \label{casenode_f}
\end{align}
where we used \eqref{L2pro} and \eqref{L2inter}.

Similarly, integrating by parts in \eqref{eq:super_edge1}, we find that if  $e\in\mathcal{E}_N^1\bigcup\mathcal{E}_N^2$, its associated patch $\omega_e$ is uniform, $\psi\in P^{e}(\omega_e)$ and $p_{\omega_e}^0\in P_0(\omega_e)$, then
$\int_{\omega_e} p_{\omega_e}^0 (\psi-I_N \psi)\,\mathrm{d}x\,\mathrm{d}y = 0$.
Hence
\begin{align}
\left|\sum_{e\in\mathcal{E}_N^1}\int_{\omega_e} f (\xi_N^e-I_N \xi_N^e)\,\mathrm{d}x\,\mathrm{d}y\right|
&\le\left|\sum_{e\in\mathcal{E}_N^1}\int_{\omega_e}
\left(f-P_{\omega_e}^0f\right)(\xi_N^e-I_N \xi_N^e)\,\mathrm{d}x\,\mathrm{d}y \right| \notag\\
&\hspace{-15mm}\le C \left(\sum_{e\in\mathcal{E}_N^1} \left\|f-P_{\omega_e}^0f\right\|_{L^2(\omega_e)}^2\right)^{1/2}
\left(\sum_{e\in\mathcal{E}_N^1}\left\|\xi_N^e-I_N \xi_N^e\right\|_{L^2(\omega_e)}^2\right)^{1/2}\notag\\
&\hspace{-15mm}\le C N^{-2}
\left(\sum_{e\in\mathcal{E}_N^1}|\xi_N^e|_{1,N}^2\right)^{1/2}, \label{caseedge1_f}
\end{align}
using a variant of \eqref{L2pro} to bound $\left\|f-P_{\omega_e}^0f\right\|_{L^2(\omega_e)}$
and the definition of $V_N^\mathcal{E}$ to deduce that  $I_N \xi_N^e = 0$.

Using Cauchy-Schwarz inequalities and \eqref{L2inter} we have
\begin{align}
\left|\sum_{e\in\mathcal{E}_N^2}\int_{\omega_e} f (\xi_N^e-I_N \xi_N^e)\,\mathrm{d}x\,\mathrm{d}y\right|
&\le C \left(\sum_{e\in\mathcal{E}_N^2} \left\|f\right\|_{L^2(\omega_e)}^2\right)^{1/2}
\left(\sum_{e\in\mathcal{E}_N^2}|\xi_N^e-I_N \xi_N^e|_{L^2(\omega_e)}^2\right)^{1/2}\notag\\
&\hspace{-15mm}\le C N^{-1} \left(\sum_{e\in\mathcal{E}_N^2}\left\|f\right\|_{L^2(\omega_e)}^2\right)^{1/2}
\left(\sum_{e\in\mathcal{E}_N^2}|\xi_N^e|_{1,N}^2\right)^{1/2}\notag\\
&\hspace{-15mm}\le C N^{-3/2}\left(\sum_{e\in\mathcal{E}_N^2}|\xi_N^e|_{1,N}^2\right)^{1/2},\label{caseedge2_f}
\end{align}
as $\text{meas}(\omega_e) = O(N^{-2})$ and there are $O(N)$ edges in $\mathcal{E}_N^2$.

Combining \eqref{casenode_f}, \eqref{caseedge1_f} and \eqref{caseedge2_f} with \eqref{sep_FEM_f} and using \eqref{basis_bound} we could get
\[
\left|\left(f,\xi_N - I_N \xi_N\right)\right|\le C N^{-3/2} \|\xi_N\|_{\eps,N}.
\]
\end{proof}

\begin{lemma}\label{trun-2}
Let $v_N, \xi_N\in V_N$ be arbitrary and set $\eta_N=u-v_N$.
Then there exists a constant~$C$, which is independent of $v_N$ and $\xi_N$, such that
\[
\eps^2\left|a_N(\eta_N,\xi_N)\right|\le C \eps\left(\sum_{K\in\mathcal{T}_N}\|\Delta_N\eta_N\|_{L^2(K)}^2\right)^{1/2} \|\xi_N\|_{\eps,N}.
\]
\end{lemma}
\begin{proof}
Two Cauchy-Schwarz inequalities and the definition of $|\cdot|_{2,N}$ give
\begin{align*}
\left|a_N(\eta_N,\xi_N)\right|
	&=\left|\sum_{K\in\mathcal{T}_N}\int_K \left(\Delta_N\eta_N\right) \left(\Delta_N \xi_N\right) \,\mathrm{d}x\,\mathrm{d}y\right|\\
&\le C\left(\sum_{K\in\mathcal{T}_N}\|\Delta_N\eta_N\|_{L^2(K)}^2\right)^{1/2}
\left(\sum_{K\in\mathcal{T}_N}\|\Delta_N\xi_N\|_{L^2(K)}^2\right)^{1/2}\\
&\le C\left(\sum_{K\in\mathcal{T}_N}\|\Delta_N\eta_N\|_{L^2(K)}^2\right)^{1/2}
|\xi_N|_{2,N}.
\end{align*}
The result now follows from the definition of $\|\xi_N\|_{\eps,N}$.
\end{proof}

\begin{lemma}\label{trun-3}
Let $v_N, \xi_N\in V_N$ be arbitrary and set $\eta_N=u-v_N$.
Then there exists a constant~$C$, which is independent of $v_N$ and $\xi_N$, such that
\begin{align*}
\left|b_N(\eta_N,\xi_N)\right|\le C \left(\sum_{K\in\mathcal{T}_N}\|\nabla_N\eta_N\|_{L^2(K)}^2\right)^{1/2} \|\xi_N\|_{\eps,N}.
\end{align*}
\end{lemma}
\begin{proof}
Two Cauchy-Schwarz inequalities and the definition of $|\xi_N|_{\eps,N}$ give
\begin{align*}
\left|b_N(\eta_N,\xi_N)\right|
&=\left|\sum_{K\in\mathcal{T}_N}\int_K\left(\nabla_N\eta_N\right)
\left(\nabla_N \xi_N\right) \,\mathrm{d}x\,\mathrm{d}y\right|\\
&\le C\left(\sum_{K\in\mathcal{T}_N}\|\nabla_N\eta_N\|_{L^2(K)}^2\right)^{1/2}
\left(\sum_{K\in\mathcal{T}_N}\|\nabla_N \xi_N\|_{L^2(K)}^2\right)^{1/2}\\
&\le C\left(\sum_{K\in\mathcal{T}_N}\|\nabla_N\eta_N\|_{L^2(K)}^2\right)^{1/2}
\|\xi_N\|_{\eps,N}.
\end{align*}
\end{proof}


%
%
\subsection{Reduced Morley interpolation}\label{sec:reducedMorley}
To finish the convergence analysis, we define an interpolation operator $\Pi_{K}^-$
that uses reduced forms of the basis functions
$p_i, q_i \,(i=1,2,3,4)$ of Section~\ref{sec:Morley}. One does not use the full basis functions because  anisotropic equivalents of Lemmas~\ref{lem4.9} and~\ref{lem4.10} are not known for them.

As in \cite{ShiXieNMPDE10}, on each mesh rectangle $K$ as in Figure~\ref{fig:DOF}
and each $v\in C^1(K)$,
define the reduced Morley interpolant $\Pi_{K}^-v(x,y)$ for $(x,y)\in K$ by
\[
\Pi_{K}^- v(x,y) := \sum_{i = 1}^{4}  v(a_i) p_i^-(x,y)
	+ \sum_{j = 1}^{4} \left(\frac{1}{|e_j|}\int_{e_j}\frac{\partial v}{\partial {n}_{e_j}}\,ds\right) q_j^- (x,y) ,
\]
where  we set
\begin{equation*}
\left\{
\begin{aligned}
&p_1^-(x,y)  := \frac{1}{8}\left[2\left(1-\frac{x-x_{c}}{h_x}\right)\left(1-\frac{y-y_{c}}{h_y}\right)
-\frac{x-x_{c}}{h_x}
-\frac{y-y_{c}}{h_y}\right],\\
&p_2^-(x,y)  := \frac{1}{8}\left[2\left(1-\frac{x-x_{c}}{h_x}\right)\left(1+\frac{y-y_{c}}{h_y}\right)
-\frac{x-x_{c}}{h_x}
+\frac{y-y_{c}}{h_y}\right],\\
&p_3^-(x,y)  := \frac{1}{8}\left[2\left(1+\frac{x-x_{c}}{h_x}\right)\left(1+\frac{y-y_{c}}{h_y}\right)
+\frac{x-x_{c}}{h_x}
+\frac{y-y_{c}}{h_y}\right],\\
&p_4^-(x,y)  := \frac{1}{8}\left[2\left(1+\frac{x-x_{c}}{h_x}\right)\left(1-\frac{y-y_{c}}{h_y}\right)
+\frac{x-x_{c}}{h_x}
-\frac{y-y_{c}}{h_y}\right],\\
&q_1^-(x,y)  := -\frac{h_y}{4}\left[-\left(\frac{y-y_{c}}{h_y}\right)^2-\frac{y-y_{c}}{h_y}+1\right], \\
&q_2^-(x,y)  = -\frac{h_x}{4}\left[-\left(\frac{x-x_{c}}{h_x}\right)^2-\frac{x-x_{c}}{h_x}+1\right],\\
&q_3^-(x,y)  := \frac{h_y}{4}\left[\left(\frac{y-y_{c}}{h_y}\right)^2-\frac{y-y_{c}}{h_y}-1\right],
\quad q_4^-(x,y)  := \frac{h_x}{4}\left[\left(\frac{x-x_{c}}{h_x}\right)^2-\frac{x-x_{c}}{h_x}-1\right].
\end{aligned}
\right.
\end{equation*}
for $(x,y)\in K$.
This set of functions is obtained by removing
all the cubic polynomials from the basis functions of the standard rectangular Morley basis of Section~\ref{sec:Morley}.
For each mesh rectangle~$K$, let $P^-(K)$ denote the space spanned by
$\{p_i^-, q_i^-: i=1,2,3,4\}$.

\begin{lemma}\label{lem4.9}
For each $K\in{\mathcal T}_N$ and any $v\in H^2(K)$,
there exist constants $C$ (which are independent of $K$ and $v$) such that
\begin{align}\label{inter2-1}
\left\|\frac{\partial(v-\Pi_{K}^- v)}{\partial x} \right\|_{L^2(K)}\le C\left(h_x\left\|\frac{\partial^2 v}{\partial x^2} \right\|_{L^2(K)}+h_y\left\|\frac{\partial^2 v}{\partial x\partial y} \right\|_{L^2(K)}\right)
\end{align}
and
\begin{align}\label{inter2-2}
\left\|\frac{\partial(v-\Pi_{K}^- v)}{\partial y} \right\|_{L^2(K)}\le C\left(h_x\left\|\frac{\partial^2 v}{\partial x\partial y} \right\|_{L^2(K)}+h_y\left\|\frac{\partial^2 v}{\partial y^2} \right\|_{L^2(K)}\right).
\end{align}
For each $K\in{\mathcal T}_N$ and any $v\in H^3(K)$,
there exist constants $C$ (which are independent of $K$ and $v$) such that
\begin{align}\label{inter1to3-1}
\left\|\frac{\partial(v-\Pi_{K}^- v)}{\partial x} \right\|_{L^2(K)}
\le C\left(
h_x^2\left\|\frac{\partial^3 v}{\partial x^3} \right\|_{L^2(K)}
+h_x h_y\left\|\frac{\partial^3 v}{\partial x^2\partial y}\right\|_{L^2(K)}
+h_y^2\left\|\frac{\partial^3 v}{\partial x\partial y^2}\right\|_{L^2(K)}
\right)
\end{align}
and
\begin{align}\label{inter1to3-2}
\left\|\frac{\partial(v-\Pi_{K}^- v)}{\partial y} \right\|_{L^2(K)}
\le C\left(
h_x^2\left\|\frac{\partial^3 v}{\partial x^2\partial y} \right\|_{L^2(K)}
+h_x h_y\left\|\frac{\partial^3 v}{\partial x\partial y^2}\right\|_{L^2(K)}
+h_y^2\left\|\frac{\partial^3 v}{\partial y^3}\right\|_{L^2(K)}
\right).
\end{align}
\end{lemma}
\begin{proof}
We combine \cite[Theorem~2.3]{ChenZhaoShiANM04} with a standard scaling argument
to prove \eqref{inter2-1} and \eqref{inter2-2},
imitating the analysis of the Wilson element in~\cite[Section~3]{ChenZhaoShiANM04}.
Let $\hat{K} := [-1,1]^2$ be the standard reference element.
Define the nodes $\hat{a}_i$ and edges $\hat{e}_j$ of $\hat K$ analogously to Figure~\ref{fig:DOF}.
Given any $\hat v\in C^1(\hat{K})$, its reduced rectangular Morley interpolant is defined
 by
\[
\hat{\Pi}_{\hat{K}}^- \hat{v}(\hat x, \hat y) :=
\sum_{i = 1}^{4} \hat{v}(\hat{a}_i)  \hat{p}_i^-(\hat x, \hat y) +
\sum_{j = 1}^{4} \left( \frac{1}{|\hat{e}_j|}\int_{\hat{e}_j}\frac{\partial \hat{v}}{\partial {n}_{\hat{e}_j}}\,ds\right)
	\hat{q}_j^-(\hat x, \hat y) \ \text{ for } (\hat x, \hat y)\in \hat K,
\]
where we set
\begin{equation*}
\left\{
\begin{aligned}
&\hat{p}_1^-(\hat x, \hat y) = \frac{1}{8}\left[2\left(1-\hat{x}\right)\left(1-\hat{y}\right)
-\hat{x}-\hat{y}\right],
\quad \hat{p}_2^-(\hat x, \hat y) = \frac{1}{8}\left[2\left(1-\hat{x}\right)\left(1+\hat{y}\right)
-\hat{x}+\hat{y}\right],\\
&\hat{p}_3^-(\hat x, \hat y) = \frac{1}{8}\left[2\left(1+\hat{x}\right)\left(1+\hat{y}\right)
+\hat{x}+\hat{y}\right],
\quad \hat{p}_4^-(\hat x, \hat y) = \frac{1}{8}\left[2\left(1+\hat{x}\right)\left(1-\hat{y}\right)
+\hat{x}-\hat{y}\right],\\
&\hat{q}_1^-(\hat x, \hat y) = -\frac{1}{4}\left(-\hat{y}^2-\hat{y}+1\right),
\quad \hat{q}_2^-(\hat x, \hat y) = -\frac{1}{4}\left(-\hat{x}^2-\hat{x}+1\right),\\
&\hat{q}_3^-(\hat x, \hat y) = \frac{1}{4}\left(\hat{y}^2-\hat{y}-1\right),
\quad \hat{q}_4^-(\hat x, \hat y) = \frac{1}{4}\left(\hat{x}^2-\hat{x}-1\right).
\end{aligned}
\right.
\end{equation*}

We check  that $\Pi_{\hat K}^-$ satisfies the hypotheses of
\cite[Theorem~2.3]{ChenZhaoShiANM04} for $|\alpha| = 1$, $p=q=2$, $l=m=0$ and $r = 3$.
Let $D^\alpha$ denote $\partial/\partial x$ or $\partial/\partial y$.
It is easy to see that $P_0(\hat K)\subset D^\alpha P^-(\hat K) = P_1(\hat K)$
and $\Pi_{\hat K}^-$ is a continuous linear mapping from $H^2(\hat K)$ to $H^1(\hat K)$.
We must verify  \cite[eqs.(2.14) and (2.15)]{ChenZhaoShiANM04}.
We shall do this for the case $\alpha = (1,0)$ (that is,  $D^\alpha = \partial/\partial x$);
the other case  $\alpha = (0,1)$ is similar.
A computation yields
\begin{align*}
\frac{\partial \hat{\Pi}_{\hat{K}}^- \hat{v}}{\partial \hat{x}} (\hat x, \hat y)
&= \frac{1}{8}\left(2\hat{y}-3\right)\left[\hat{v}(\hat{a}_1)-\hat{v}(\hat{a}_4)\right]
+\frac{1}{8}\left(2\hat{y}+3\right)\left[\hat{v}(\hat{a}_3)-\hat{v}(\hat{a}_2)\right]\\
&\qquad +\frac{1}{8}\left(2\hat{x}+1\right)\int_{-1}^1-\frac{\partial \hat{v}}{\partial \hat{x}}(-1,\hat{y})d\hat{y}
+\frac{1}{8}\left(2\hat{x}-1\right)\int_{-1}^1\frac{\partial \hat{v}}{\partial \hat{x}}(1,\hat{y})d\hat{y}\\
&\hspace{-15mm}=\frac{\hat{y}}{4}\left[\int_{-1}^1\frac{\partial \hat{v}}{\partial \hat{x}}(\hat{x},1)d\hat{x}-\int_{-1}^1\frac{\partial \hat{v}}{\partial \hat{x}}(\hat{x},-1)d\hat{x}\right]
+\frac{\hat{x}}{4}\left[\int_{-1}^1\frac{\partial \hat{v}}{\partial \hat{x}}(1,\hat{y})d\hat{y}-\int_{-1}^1\frac{\partial \hat{v}}{\partial \hat{x}}(-1,\hat{y})d\hat{y}\right]\\
&\hspace{-10mm} +\frac{1}{8}\left[3\int_{-1}^1\frac{\partial \hat{v}}{\partial \hat{x}}(\hat{x},1)d\hat{x}+3\int_{-1}^1\frac{\partial \hat{v}}{\partial \hat{x}}(\hat{x},-1)d\hat{x}-\int_{-1}^1\frac{\partial \hat{v}}{\partial \hat{x}}(1,\hat{y})d\hat{y}-\int_{-1}^1\frac{\partial \hat{v}}{\partial \hat{x}}(-1,\hat{y})d\hat{y}\right]\\
&\hspace{-15mm}=\frac{\hat{y}}{4}\int_{\hat{K}}\frac{\partial^2 \hat{v}}{\partial \hat{y}\partial \hat{x}}d\hat{x}d\hat{y}
+\frac{\hat{x}}{4}\int_{\hat{K}}\frac{\partial^2 \hat{v}}{\partial \hat{x}^2}d\hat{x}d\hat{y}
+\left[\frac{3}{8}\int_{\hat{K}}\frac{\partial}{\partial \hat{y}}\left(\hat{y}\frac{\partial \hat{v}}{\partial \hat{x}}\right)\,d\hat{x}\,d\hat{y}
-\frac{1}{8}\int_{\hat{K}}\frac{\partial}{\partial \hat{x}}\left(\hat{x}\frac{\partial \hat{v}}{\partial \hat{x}}\right)
\,d\hat{x}\,d\hat{y}\right]\\
&:=\sum_{i = 1}^3\beta_i(\hat{v}) \hat{q}_i(\hat x, \hat y),
\end{align*}
where $\beta_i(\hat{v}) := F_i(\frac{\partial \hat{v}}{\partial \hat{x}})$ for $i=1,2,3$ with
$F_1(\hat{w}) := \frac{1}{4} \int_{\hat{K}} \frac{\partial \hat{w}}{\partial \hat{y}}\,d\hat{x}\,d\hat{y},
\ F_2(\hat{w}) := \frac{1}{4}\int_{\hat{K}} \frac{\partial \hat{w}}{\partial \hat{x}} \,d\hat{x}\,d\hat{y}$,
$$
F_3(\hat{w}) := \frac{3}{8}\int_{\hat{K}}\frac{\partial\left(\hat{y}\hat{w}\right)}{\partial \hat{y}}\,d\hat{x}\,d\hat{y}
-\frac{1}{8}\int_{\hat{K}}\frac{\partial\left(\hat{x}\hat{w}\right)}{\partial \hat{x}}\,d\hat{x}\,d\hat{y},
$$
and $\hat{q}_1(\hat x, \hat y) := \hat{y}, \, \hat{q}_2(\hat x, \hat y) := \hat{x}, \,
\hat{q}_3(\hat x, \hat y) := 1$.
Clearly  $\{\hat{q}_1, \hat{q}_2, \hat{q}_3\}$ is a basis for
$\hat{D}^{\alpha} P^-(\hat{K}) = P_1(\hat{K})$, so
we have now verified  \cite[eq.(2.14)]{ChenZhaoShiANM04}.
To verify \cite[eq.(2.15)]{ChenZhaoShiANM04} is straightforward: Cauchy-Schwarz inequalities
yield $\left|F_i(\hat{w})\right|\le \hat{C}\|\hat{w}\|_{H^1(\hat{K})}$ for $i = 1, 2, 3$.

One can now invoke \cite[Theorem~2.3]{ChenZhaoShiANM04} for
$|\alpha| = 1$, $p=q=2$, $l=m=0$ and $r = 3$.
Then a standard scaling argument gives \eqref{inter2-1} and \eqref{inter2-2}.

The interpolation error estimates \eqref{inter1to3-1} and \eqref{inter1to3-2} can be proved similarly,
except that now $l=1$ (we had $l=0$ for \eqref{inter2-1} and \eqref{inter2-2}),
while the other parameters are unchanged: $|\alpha| = 1$, $p=q=2$, $m=0$ and $r = 3$.
It is easy to see that $P_1(\hat K)\subset D^\alpha P^-(\hat K) = P_1(\hat K)$ and $\Pi_{\hat K}^-$
is a continuous linear mapping from $H^3(\hat K)$ to $H^1(\hat K)$.
The verification of \cite[eq.(2.14)]{ChenZhaoShiANM04} is same as before
and the verification of \cite[eq.(2.15)]{ChenZhaoShiANM04} is similar
since $\left|F_i(\hat{w})\right|\le \hat{C}\|\hat{w}\|_{H^1(\hat{K})}\le \hat{C}\|\hat{w}\|_{H^2(\hat{K})}$ for $i = 1, 2, 3$.
Then we invoke \cite[Theorem~2.3]{ChenZhaoShiANM04}, followed by a standard scaling argument,
thereby obtaining \eqref{inter1to3-1} and \eqref{inter1to3-2}.
\end{proof}

The next result is the analogue of Lemma~\ref{lem4.9} for 2nd-order derivatives.

\begin{lemma}\label{lem4.10}
For each $K\in{\mathcal T}_N$ and any $v\in H^3(K)$, there exist constants $C$ (independent of $K$ and $v$) such that
\begin{align}\label{inter3-1}
\left\|\frac{\partial^2(v-\Pi_{K}^- v)}{\partial x^2} \right\|_{L^2(K)}\le C\left(h_x\left\|\frac{\partial^3 v}{\partial x^3} \right\|_{L^2(K)}+h_y\left\|\frac{\partial^3 v}{\partial x^2\partial y} \right\|_{L^2(K)}\right),
\end{align}
\begin{align}\label{inter3-2}
\left\|\frac{\partial^2(v-\Pi_{K}^- v)}{\partial x\partial y} \right\|_{L^2(K)}\le C\left(h_x\left\|\frac{\partial^3 v}{\partial x^2\partial y} \right\|_{L^2(K)}+h_y\left\|\frac{\partial^3 v}{\partial x\partial y^2} \right\|_{L^2(K)}\right)
\end{align}
and
\begin{align}\label{inter3-3}
\left\|\frac{\partial^2(v-\Pi_{K}^- v)}{\partial y^2} \right\|_{L^2(K)}\le C\left(h_x\left\|\frac{\partial^3 v}{\partial x\partial y^2} \right\|_{L^2(K)}+h_y\left\|\frac{\partial^3 v}{\partial y^3} \right\|_{L^2(K)}\right).
\end{align}
\end{lemma}
\begin{proof}
Apply the first inequality of  \cite[eq.(12)]{ShiXieNMPDE10} on the reference element $\hat K= [-1,1]^2$,
then use a standard scaling argument.
\end{proof}

\subsection{Error bound on the Shishkin mesh}\label{sec:errorShishkin}
We can now complete the error analysis of our numerical method.

\begin{theorem}\label{thm:cgce}
Let $u$ be the solution of \eqref{BVP}  and $u_N$ the solution of \eqref{BVP_FEM}.
Let Assumption \ref{ass:S} be satisfied.
Then there exists a constant $C$ (which is independent of $\eps$ and $N$) such that
\begin{equation*}
\left\|u-u_N\right\|_{\eps,N}\le C \left(\eps^{1/2}N^{-1}+\eps N^{-1}\ln N + N^{-3/2} \right).
\end{equation*}
\end{theorem}
\begin{proof}
We imitate the argument of  \cite[Theorem 1]{ShiXieNMPDE10}.
Define $v_N\in V_N$ by $v_N = \Pi_{K}^- u$ on each $K\in\mathcal{T}_N$.
Set $\xi_N = u_N-v_N$. and $\eta_N = u-v_N$.
Recalling the error equation \eqref{errorequation}, we estimate the terms on its right-hand side using
Lemmas \ref{consis1-3}, \ref{consis2-3} and \ref{trun-1}, \ref{trun-4E}, \ref{trun-4S}
and the definitions of~$a_N$ and~$b_N$: this gives
\begin{align*}
\|\xi_N\|_{\eps,N}^2&\le C \left[\eps^2 a_N(\xi_N,\xi_N) + b_N(\xi_N, \xi_N)\right]  \notag\\
&\le C \left[\eps^{1/2}N^{-1}+\eps N^{-1}\ln N + N^{-3/2}
	+\eps\left(\sum_{K\in\mathcal{T}_N}\|\Delta_N\eta_N\|_{L^2(K)}^2\right)^{1/2} \right.  \\
&\hspace{20mm} \left. +\left(\sum_{K\in\mathcal{T}_N}\|\nabla_N\eta_N\|_{L^2(K)}^2\right)^{1/2}\right] 	\|\xi_N\|_{\eps,N},
\end{align*}
for some constants $C$.
But $u-u_N = u-v_N + v_N - u_N = \eta_N - \xi_N$, so
\begin{align}
&\left\|u-u_N\right\|_{\eps,N}\le \left\|\eta_N\right\|_{\eps,N}
+C\left(\eps^{1/2}N^{-1}+\eps N^{-1}\ln N + N^{-3/2} \right)\notag\\
&\qquad +C\left[\eps\left(\sum_{K\in\mathcal{T}_N}\|\Delta_N(u-\Pi_{K}^- u)\|_{L^2(K)}^2\right)^{1/2}
+\left(\sum_{K\in\mathcal{T}_N}\|\nabla_N(u-\Pi_{K}^- u)\|_{L^2(K)}^2\right)^{1/2}\right] \notag\\
&\quad \le C\left(\eps^{1/2}N^{-1}+\eps N^{-1}\ln N + N^{-3/2} \right)\notag\\
&\qquad +C\left[\eps\left(\sum_{K\in\mathcal{T}_N}\|\Delta_N(u-\Pi_{K}^- u)\|_{L^2(K)}^2\right)^{1/2}
+\left(\sum_{K\in\mathcal{T}_N}\|\nabla_N(u-\Pi_{K}^- u)\|_{L^2(K)}^2\right)^{1/2}\right],
\label{halfresult}
\end{align}
from the definitions of $\left\|\cdot\right\|_{\eps,N}$ and $\eta_N$.

The bounds of Lemma~\ref{lem4.10} yield
\begin{align*}
&\left(\sum_{K\in\mathcal{T}_N}
\|\Delta_N(u-\Pi_{K}^- u)\|_{L^2(K)}^2\right)^{1/2}\\
&\le C \bigg[\sum_{K\in\mathcal{T}_N}\bigg(h_x^2\left\|\frac{\partial^3 u}{\partial x^3} \right\|_{L^2(K)}^2
+h_y^2\left\|\frac{\partial^3 u}{\partial x^2\partial y} \right\|_{L^2(K)}^2
+h_x^2\left\|\frac{\partial^3 u}{\partial x^2\partial y} \right\|_{L^2(K)}^2
+h_y^2\left\|\frac{\partial^3 u}{\partial x\partial y^2} \right\|_{L^2(K)}^2\\
&\quad+h_x^2\left\|\frac{\partial^3 u}{\partial x\partial y^2} \right\|_{L^2(K)}^2
+h_y^2\left\|\frac{\partial^3 u}{\partial y^3} \right\|_{L^2(K)}^2\bigg)\bigg]^{1/2}.
\end{align*}
To estimate each of these terms, one can proceed similarly to the proof of Lemma~\ref{consis1-3} above.
Specifically, the worst-behaved terms are
$h_x^2\left\|\frac{\partial^3 u}{\partial x\partial y^2}\right\|_{L^2(K)}^2$
and $h_y^2\left\|\frac{\partial^3 u}{\partial x^2\partial y}\right\|_{L^2(K)}^2$.
Comparing $h_x^2\left\|\frac{\partial^3 u}{\partial x\partial y^2}\right\|_{L^2(K)}^2$
with the term $h_x^2\left\|\frac{\partial^4 u}{\partial x\partial y^3}\right\|_{L^2(K)}^2$ in Lemma~\ref{consis1-3},
here the order of the derivative is one less, which effectively removes a factor $\eps^{-1}$ from the calculation leading to~\eqref{consisu}, so instead of~\eqref{consisu} one obtains
\begin{equation}\label{similar2}
\eps\left(\sum_{K\in\mathcal{T}_N}
\|\Delta_N(u-\Pi_{K}^- u)\|_{L^2(K)}^2\right)^{1/2}\le C \left(\eps^{1/2}N^{-1}+\eps N^{-1}\ln N+N^{-2}\right).
\end{equation}
Next, for the boundary and corner layer components, the bounds of Lemma~\ref{lem4.9} give
\begin{align}
&\left(\sum_{K\in\mathcal{T}_N}\|\nabla_N(u^E-\Pi_{K}^- u^E)\|_{L^2(K)}^2\right)^{1/2} \notag\\
&\le C \bigg[\sum_{K\in\mathcal{T}_N}
	\bigg(h_x^2\left\|\frac{\partial^2 u^E}{\partial x^2} \right\|_{L^2(K)}^2
	+h_y^2\left\|\frac{\partial^2 u^E}{\partial x\partial y} \right\|_{L^2(K)}^2
	+h_x^2\left\|\frac{\partial^2 u^E}{\partial x\partial y} \right\|_{L^2(K)}^2
	+h_y^2\left\|\frac{\partial^2 u^E}{\partial y^2} \right\|_{L^2(K)}^2\bigg)\bigg]^{1/2}  \notag\\
&\le C \left(\eps^{1/2}N^{-1}+\eps N^{-1}\ln N + N^{-2}\right)\label{similar3}
\end{align}
by \eqref{similar1}.
For the smooth component, by \eqref{inter1to3-1} and \eqref{inter1to3-2} we have
\begin{align}\label{similar4}
\left(\sum_{K\in\mathcal{T}_N}\|\nabla_N(S-\Pi_{K}^- S)\|_{L^2(K)}^2\right)^{1/2}\le CN^{-2}.
\end{align}

Substituting \eqref{similar2}, \eqref{similar3} and \eqref{similar4} into \eqref{halfresult}, we are done.
\end{proof}

It is well known that in the analysis of singularly perturbed problems on layer-adapted meshes
such as the Shishkin mesh, the most troublesome regime is when $\eps\approx N^{-1}$.
For this regime, Theorem~\ref{thm:cgce} gives the following result.

\begin{corollary}\label{cor:cgce}
Let $u$ be the solution of \eqref{BVP}  and $u_N$ the solution of \eqref{BVP_FEM}.
Let Assumption \ref{ass:S} be satisfied.
If $\eps\approx N^{-1}$, then there exists a constant $C$ (which is independent of $\eps$ and $N$) such that
\begin{equation*}
\left\|u-u_N\right\|_{\eps,N}\le C N^{-3/2}.
\end{equation*}
\end{corollary}

\begin{remark}[Adini versus Morley]\label{AvM}
The Adini  element is a well-known example of  a nonconforming finite element
that is suitable for our singularly perturbed problem~\eqref{BVP}.
For the Adini element on an appropriate Shishkin mesh,
one has the error bound \cite[Corollary~4.1]{MengStynes19}
\[
\|u-u_N\|_{\eps,N} \le
	\begin{cases}
	C\left(\eps^{1/2} + N^{-3} \right) &\text{if } \eps\le N^{-1}, \\
	C\left( \eps^{1/2}(N^{-1}\ln N)^2 +  \eps^{-3/2}N^{-2}  \right)  &\text{if } \eps> N^{-1}.
	\end{cases}
\]
Thus in the most challenging regime when $\eps \approx N^{-1}$,
the Adini element obtains $O(N^{-1/2})$ convergence,
but Corollary~\ref{cor:cgce} shows that the rectangular Morley element attains $O(N^{-3/2})$ convergence.
\end{remark}

\begin{remark}\label{rem:sharper}
An inspection of the entire analysis leading to the result of Theorem~\ref{thm:cgce}
shows that we are close to proving the sharper error bound
\begin{equation}\label{sharper}
\left\|u-u_N\right\|_{\eps,N}\le C \left(\eps^{1/2}N^{-1}+\eps N^{-1}\ln N + N^{-2} \right),
\end{equation}
which would agree with our numerical results in Section~\ref{sec:numer}.
The weaker term $N^{-3/2}$ in Theorem~\ref{thm:cgce} appears only because of the
two inequalities \eqref{caseedge2} and~\eqref{caseedge2_f}---all other steps
of the analysis yield terms in~\eqref{sharper}. Unfortunately we are unable at present to improve
 \eqref{caseedge2} and~\eqref{caseedge2_f}.
\end{remark}

\section{Numerical experiments}\label{sec:numer}
To test the accuracy of our finite element method we present three numerical examples, one with a known solution that displays typical layer behaviour and two others with unknown solutions.

\begin{example}\label{exa:1}
In \eqref{BVP} choose $c(x,y) \equiv 1$ and choose $f(x,y)$ such that  the exact solution
is $u(x,y) = g(x)h(y)$, where
\begin{equation*}
g(x) := \frac{1}{2}\left[\sin(\pi x)+\frac{\pi\eps}{1-e^{-1/\eps}}\left(e^{-x/\eps}+e^{(x-1)/\eps}-1-e^{-1/\eps}\right)\right]
\end{equation*}
and
\begin{equation*}
h(y) := 2y(1-y^2)+\eps\left[ld(1-2y)-3\frac{q}{l}+\left(\frac{3}{l}-d\right)e^{-y/\eps}+\left(\frac{3}{l}+d\right)e^{(y-1)/\eps}\right]
\end{equation*}
with $l = 1-e^{-1/\eps}$, $q = 2-l$ and $d = 1/(q-2\eps l)$.
\end{example}

The derivatives of $u$ in this example match the bounds of Assumption~\ref{ass:S}.
The solution computed by our finite element method~\eqref{BVP_FEM} is $u_N$. Table~\ref{Table:numerresust1} displays the errors $\|u-u_N\|_{\eps,N}$
for various values of $\eps$ and $N$, and the corresponding convergence rates
$$
r = \frac{\ln(\|u-u_N\|_{\eps,N})-\ln(\|u-u_{2N}\|_{\eps,2N})}{\ln (2N) - \ln (N)},
$$
where the convergence is assumed to be $O(N^{-r})$ for each fixed value of~$\eps$,
i.e.,  the convergence rate is measured along each row of the table.


\begin{table}[htbp]
\centering
\begin{tabular}{|c|c|c|c|c|}
\hline
$\eps$    & $N$ = 16    & $N$ = 32       & $N$ = 64        & $N$ = 128         \\\hline
1.0e-00    & 2.08e-03 & 1.02e-03  & 5.07e-04    & 2.52e-04       \\\hline
    &1.02  &1.01   &1.00     &        \\\hline
1.0e-01    & 2.46e-02 & 1.12e-02  & 5.33e-03    & 2.59e-03       \\\hline
    &1.13  &1.07   &1.03     &        \\\hline
1.0e-02    & 3.79e-02 & 1.19e-02  & 4.46e-03    & 1.90e-03      \\\hline
    &1.66  &1.42   &1.23    &        \\\hline
1.0e-03    & 3.52e-02 & 1.01e-02  & 3.23e-03    & 1.14e-03      \\\hline
    &1.79  &1.65   &1.50    &        \\\hline
1.0e-04    & 3.14e-02 & 8.44e-03  & 2.34e-03    & 7.05e-04      \\\hline
    &1.89  &1.84   &1.73     &        \\\hline
1.0e-05    & 2.99e-02 & 7.77e-03  & 2.02e-03    & 5.43e-04      \\\hline
    &1.94  &1.94   &1.89     &        \\\hline
1.0e-06    & 2.94e-02 & 7.54e-03  & 1.91e-03     & 4.92e-04       \\\hline
    &1.96  &1.97   &1.96     &        \\\hline
1.0e-07    & 2.92e-02 & 7.47e-03  & 1.88e-03     & 4.75e-04       \\\hline
    &1.96  &1.98   &1.98     &        \\\hline
1.0e-08    & 2.92e-02 & 7.45e-03  & 1.87e-03     & 4.70e-04       \\\hline
    &1.97  &1.99   &1.99     &        \\\hline
\end{tabular}
\caption{Numerical results for Example~\ref{exa:1}}
\label{Table:numerresust1}
\end{table}

To test the case $\eps \approx N^{-1}$ that was discussed in Remark~\ref{AvM},
we choose $\eps = N^{-1}$ in Example \ref{exa:1}.
Table~\ref{Table:numerresust1-2} displays the errors $\|u-u_N\|_{N^{-1},N}$
for various values of $N$, and the corresponding convergence rates
$$
r = \frac{\ln(\|u-u_N\|_{N^{-1},N})-\ln(\|u-u_{2N}\|_{(2N)^{-1},2N})}{\ln (2N) - \ln (N)},
$$
where the convergence is assumed to be $O(N^{-r})$.
The table shows that our method attains $O(N^{-3/2})$ convergence, in agreement with Corollary~\ref{cor:cgce}.

\begin{table}[htbp]
\centering
\begin{tabular}{|c|c|c|c|c|}
\hline
$N$ = 16    & $N$ = 32    & $N$ = 64       & $N$ = 128        & $N$ = 256         \\\hline
2.66e-02    & 1.18e-02 & 4.76e-03  & 1.79e-03    & 6.50e-04       \\\hline
1.16    &1.31  &1.41   &1.46     &        \\\hline
\end{tabular}
\caption{Numerical results for Example~\ref{exa:1} with $\eps = N^{-1}$}
\label{Table:numerresust1-2}
\end{table}

\begin{example}\label{exa:2}
In \eqref{BVP} we impose the transverse loading $f(x,y)=100(1-x-y+2xy)(x+y-2xy)$,
which is the same as the numerical example of~\cite[equation~(15)]{Franz16},
and choose $c(x,y)= 3 + (1 + x)y^2 + (2-y)e^x$.
The exact solution of this problem is unknown.
\end{example}

\begin{example}\label{exa:3}
In \eqref{BVP} we impose the transverse loading $f(x,y) = 2\pi^2[1-\cos(2\pi x)\cos(2\pi y)]$ and choose $c(x,y) = 1$;
this numerical example was tested in~\cite[Example 5.3]{Xie13} and \cite[Example 5.3]{HanHuangNMPDE12}.
Its exact solution is unknown.
\end{example}

Numerical results for Example~\ref{exa:2} and \ref{exa:3} are shown in Table~\ref{Table:numerresust2} and \ref{Table:numerresust3}.
As the exact solution is unknown, we use the double mesh principle \cite[Chapter 8]{FHMOSbook} to estimate the errors and rates of convergence.
To be specific, in the table $\|u-u_N\|_{\eps,N}$ is replaced by $\|\tilde{u}_{2N}-u_N\|_{\eps,N}$, where $\tilde{u}_{2N}$ is the solution computed by the same method on a Shishkin-type mesh that is constructed by bisecting each element $K\in\mathcal{T}_N$ in the $x$ and $y$ directions, i.e., dividing each $K$ into 4 rectangles of equal size. Similarly, in the above formula for estimating the numerical rate $r$ of convergence, we replace $u-u_M$ by $\tilde u_{2M}-u_M$ for $M=N,2N$.


\begin{table}[htbp]
\centering
\begin{tabular}{|c|c|c|c|c|}
\hline
$\eps$    & $N$ = 16    & $N$ = 32       & $N$ = 64        & $N$ = 128         \\\hline
1.0e-00    & 1.45e-02 & 7.15e-03  & 3.54e-03    & 1.76e-03       \\\hline
    &1.02  &1.01   &1.00     &        \\\hline
1.0e-01    & 6.17e-02 & 3.13e-02  & 1.55e-02    & 7.67e-03       \\\hline
    &0.97  &1.01   &1.01     &        \\\hline
1.0e-02    & 3.94e-02 & 1.71e-02  & 8.66e-03    & 4.75e-03      \\\hline
    &1.20  &0.98   &0.86    &        \\\hline
1.0e-03    & 3.05e-02 & 1.01e-02  & 3.99e-03    & 1.81e-03      \\\hline
    &1.58  &1.35   &1.13    &        \\\hline
1.0e-04    & 2.69e-02 & 7.61e-03  & 2.36e-03    & 8.53e-04      \\\hline
    &1.82  &1.68   &1.47     &        \\\hline
1.0e-05    & 2.57e-02 & 6.78e-03  & 1.84e-03    & 5.45e-04      \\\hline
    &1.92  &1.87   &1.76     &        \\\hline
1.0e-06    & 2.53e-02 & 6.52e-03  & 1.68e-03     & 4.47e-04       \\\hline
    &1.95  &1.95   &1.91     &        \\\hline
1.0e-07    & 2.52e-02 & 6.44e-03  & 1.63e-03     & 4.16e-04       \\\hline
    &1.96  &1.98   &1.96     &        \\\hline
1.0e-08    & 2.51e-02 & 6.41e-03  & 1.61e-03     & 4.07e-04       \\\hline
    &1.97  &1.98   &1.98     &        \\\hline
\end{tabular}
\caption{Numerical results for Example~\ref{exa:2}}
\label{Table:numerresust2}
\end{table}

Tables~\ref{Table:numerresust1},~\ref{Table:numerresust2} and~\ref{Table:numerresust3} seem to
indicate that the errors in the computed solutions obey the bound
\[
\left\|u-u_N\right\|_{\eps,N}\le C \left(\eps^{1/2}N^{-1}+\eps N^{-1}\ln N + N^{-2} \right)
\]
of equation \eqref{sharper}. Thus Theorem~\ref{thm:cgce} above appears not to be sharp;
this weakness in the analysis was discussed in Remark~\ref{rem:sharper}.

\begin{table}[htbp]
\centering
\begin{tabular}{|c|c|c|c|c|}
\hline
$\eps$    & $N$ = 16    & $N$ = 32       & $N$ = 64        & $N$ = 128         \\\hline
1.0e-00    & 1.22e-02 & 6.01e-03  & 2.98e-03    & 1.48e-03       \\\hline
    &1.02  &1.01   &1.00     &        \\\hline
1.0e-01    & 1.00e-01 & 4.65e-02  & 2.21e-02    & 1.07e-02       \\\hline
    &1.11  &1.07   &1.03     &        \\\hline
1.0e-02    & 1.59e-01 & 4.48e-02  & 1.69e-02    & 7.95e-03      \\\hline
    &1.82  &1.40   &1.08    &        \\\hline
1.0e-03    & 1.82e-01 & 5.08e-02  & 1.50e-02    & 4.96e-03      \\\hline
    &1.84  &1.76   &1.59    &        \\\hline
1.0e-04    & 1.79-01 & 4.77e-02  & 1.27e-02    & 3.59e-03      \\\hline
    &1.90  &1.90   &1.82     &        \\\hline
1.0e-05    & 1.77e-01 & 4.65e-02  & 1.20e-02    & 3.13e-03      \\\hline
    &1.93  &1.95   &1.93     &        \\\hline
1.0e-06    & 1.76e-01 & 4.61e-02  & 1.17e-02     & 2.99e-03       \\\hline
    &1.93  &1.97   &1.97     &        \\\hline
1.0e-07    & 1.76e-01 & 4.60e-02  & 1.16e-02     & 2.94e-03       \\\hline
    &1.93  &1.97   &1.98     &        \\\hline
1.0e-08    & 1.76e-01 & 4.60e-02  & 1.16e-02     & 2.93e-03       \\\hline
    &1.94  &1.97   &1.99     &        \\\hline
\end{tabular}
\caption{Numerical results for Example~\ref{exa:3}}
\label{Table:numerresust3}
\end{table}

\bibliographystyle{plain}
\bibliography{Morley}

\end{document}